\documentclass[oneside,a4paper,leqno]{article}
\usepackage{mystyle_arxiv}
\DeclareMathOperator{\grad}{\nabla}

\DeclareMathOperator{\rg}{rg}
\DeclareMathOperator{\dom}{dom}
\DeclareMathOperator*{\argmin}{arg\,min}

\newcommand{\dd}{\, \mathrm{d}}

\newcommand{\dx}{\dd x}
\newcommand{\dy}{\dd y}
\newcommand{\dz}{\dd z}
\newcommand{\dxi}{\dd \xi}
\newcommand{\dv}{\dd v}
\newcommand{\conv}{\star}

\newcommand{\wR}{R} 
\newcommand{\mF}{\mathcal{F}} %
\newcommand{\yn}{{y^n}}
\newcommand{\un}{{u^n}}
\newcommand{\pn}{{p^n}}
\newcommand{\sigman}{{\sigma^n}}

\newcommand{\R}{\mathbb{R}}
\newcommand{\N}{\mathbb{N}}

\newcommand{\Testfunctions}{\mathcal{D}(\Omega)}

\newcommand{\Lp}{{L^p (\Omega)}}

\newcommand{\Lq}{{{L^q (\Omega)}}}

\newcommand{\Ltwo}{{L^2 (\Omega)}}

\newcommand{\Ltworn}{{L^2 (\R^N)}}
\newcommand{\Linfty}{{L^\infty (\Omega)}}

\newcommand{\He}{{H^1 (\Omega)}}

\newcommand{\Hen}{H^{1}_0(\Omega)}

\newcommand{\Hs}{{H^s (\Omega)}}

\newcommand{\Hend}{H^{-1}(\Omega)}

\newcommand{\Wad}{{W_{ad}}}
\newcommand{\Uad}{{U_{ad}}}
\newcommand{\Fad}{{F_{ad}}}

\newcommand{\vertiii}[1]{{\left\vert\kern-0.25ex\left\vert\kern-0.25ex\left\vert #1 
    \right\vert\kern-0.25ex\right\vert\kern-0.25ex\right\vert}}

\newcommand{\ydelta}{{y_\delta}}

\newcommand{\diam}{d}

\newcommand{\m}[1]{\mathcal{#1}}

\newcommand{\dual}{'}

\newcommand{\yex}{y^\dagger} 
\newcommand{\uex}{u^\dagger}

\newcommand{\ie}{i.e. }
\newcommand{\eg}{e.g. }

\newcommand{\ale}{a.e.\ }

\newcommand{\Lin}{\mathcal{L}} 
\newcommand{\ubar}[1]{\underaccent{\bar}{#1}}
\usepackage{amsopn}

\title{Learning nonlocal regularization operators}

\author{Gernot Holler\thanks{Institute of Mathematics and Scientific Computing, University of Graz,
  Heinrichstr. 36, 8010 Graz, Austria, \email{gernot.holler@uni-graz.at}. The author gratefully acknowledges support by the International Research Training Group IGDK 1754 „Optimization and Numerical Analysis for Partial Differential Equations with Nonsmooth Structures“, funded by the German Research Council (DFG) and the Austrian Science Fund (FWF):[W 1244-N18].} 
\and Karl Kunisch \thanks{Institute of Mathematics and Scientific Computing, University of Graz, Heinrichstr. 36, 8010 Graz, Austria, and Radon
Institute, Austrian Academy of Sciences, Linz, Austria, \email{karl.kunisch@uni-graz.at}. The author gratefully acknowledges partial support by the  ERC advanced grant 668998 (OCLOC) under the EU's H2020
research program.}}

\begin{document}

\maketitle
\begin{abstract}
A learning approach for determining which operator from a class of nonlocal operators is optimal for the regularization of an inverse problem is investigated. The considered class of nonlocal operators is motivated by the use of squared fractional order Sobolev seminorms as regularization operators. First fundamental results from the theory of regularization with local operators are extended to the nonlocal case. Then a framework based on a bilevel optimization strategy is developed which allows to choose nonlocal regularization operators from a given class which i) are optimal with respect to a suitable performance measure on a training set, and ii) enjoy particularly favorable properties. Results from numerical experiments are also provided.
\end{abstract}

%

\begin{quote}
\end{quote}
\begin{quote}
\end{quote}

\section{Introduction}
In this work we discuss the use of a family of nonlocal energy seminorms for the regularization of inverse problems governed by partial differential equations. The archetypes for the considered family are Sobolev seminorms $|u|_{\Hs}$ of fractional order $s \in (0,1)$. The corresponding regularized inverse problems are
\begin{equation}\label{prob:reg_Sobolev_seminorm}
\min_{u \in \dom (S) \cap \Hs } \| S(u) - \ydelta \|_\Ltwo^2 + \nu |u|_{\Hs}^2.
\end{equation}
Here $S \colon \dom (S)  \subseteq \Ltwo \to \Ltwo$ is the given forward operator, $\nu >0$ is a regularization parameter, and $\ydelta \in \Ltwo$ is the given measurement. The considered family of nonlocal energy seminorms $(|\cdot|_{\gamma,s})_{\gamma \in \Wad}$  will differ from Sobolev seminorms only by additional weighting terms $\gamma \in \Wad$. A precise definition of the nonlocal energy seminorms and the set of admissible weights $\Wad$ will be given in \Cref{sec:preliminaries} below. The corresponding regularized inverse problem for a particular weight $\gamma \in \Wad$ is
\begin{equation}
\min_{u \in \dom (S) \cap \Hs} \| S(u) - \ydelta \|_\Ltwo^2 + | u |_{\gamma,s}^2.
\end{equation}

\paragraph{The learning problem}
To determine which element from this family of nonlocal energy seminorms is particularly suitable for a given problem we use a learning approach: We assume to be given ground truth data and noisy measurements, \ie a set $(\yex_i, \uex_i, \ydelta_i)_{1 \leq i \leq N_\text{Train}}$ such that
\begin{equation*}
S(\uex_i) =\yex_i,
\end{equation*}
and $\ydelta_i$ are noisy measurements of $\yex_i$ for $1 \leq i \leq N_{\text{Train}}$. We then determine a weight $\gamma^*$ such that solutions to the corresponding inverse problems represent the ground truth data particularly well. This is done by choosing $\gamma^* \in \Wad$ as a solution to
\begin{equation}\label{prob:bilevelsimple_introduction}
\tag{BP}
\begin{split}
\min_{\gamma \in \Wad, u_i \in H^{s}(\Omega)} \frac{1}{2 N_\text{Train}} \sum\limits_{i=1}^{N_{\text{Train}}} \| &u_i - \uex_i \|_{\Ltwo}^2 + R(\gamma) \quad  \\ \text{subject to} \quad &u_i \in \argmin_{u \in \dom(S) \cap \Hs} \| S(u) - \ydelta_i\|^2_{\Ltwo} + | u |_{\gamma,s}^2.
\end{split}
\end{equation}
Here, $R \colon \Wad \to \R$ is an added regularization operator. We will favor the choice $R$ as the $L^1$ norm. This has the effect that nonlocality is only utilized if its effect is sufficiently strong, otherwise it is set to zero. As a side effect of this procedure, we obtain that in the regularized inverse problem the system matrices, which tend to be densely populated in the context of fractional order regularization, in fact become more sparse. Except for the numerical experiments, we only consider the case $N_\text{Train} = 1$. However, generalization of the analytical results to the case of multiple data vectors is straightforward using product spaces.

This paper is organized as follows. In Section 2 the necessary background is provided, and a stability property for solutions to Poisson-type nonlocal equations, which will be frequently needed throughout this work, is derived. Moreover, the class of weights considered in this work is introduced. Section 3 is concerned with the case of a linear forward problem. After deriving some basic properties of the regularized inverse problem, existence of solutions to the learning problem is proven and an optimality system is derived. In Section 4 we discuss the nonlinear case. After providing some results, which can be applied to general nonlinear functions, we discuss in detail the problem of estimating the convection term in an elliptic PDE. Finally, in Section 5 results from numerical experiments are presented which demonstrate the feasibility of our approach. 

\paragraph{Related work}
Note that \eqref{prob:bilevelsimple_introduction} is a bilevel optimization problem, \ie an optimization problem, where the constraint involves another optimization problem (referred to as the lower level problem). A standard reference on bilevel optimization is \cite{Dempe2002}. Nonlocal operators have recently received a significant amount of attention in the literature, see \eg\cite{Gunzburger2010,Du2012, DElia2013, Alali2015}. A learning problem for determining optimal filter parameters for nonlocal regularization operators in the context
of image denoising problems was recently investigated in \cite{DElia2019}. As a particular instance of nonlocal regularization operators, fractional-type regularization operators are considered in
\cite{Antil2018b, Antil2019}. In terms of learning theory, the problem of learning regularization operators can be viewed as a supervised learning problem. The problem of choosing regularization operators from a parametrized class of functions based on training data, is studied in \cite{haber2003}. Optimal spectral filters for finite dimensional inverse problems are learned in \cite{chung2011designing}. Learning strategies for choosing regularization parameters in the context of multi-penalty Tikhonov regularization are investigated \eg in \cite{Kunisch_Bilevel2013, DELOSREYES2016464,chung2017,Holler2018}. The problem of learning the discrepancy function is considered in \cite{de2013image}. In many of the mentioned references, the lower level problem is not differentiable, which in turn complicates the derivation of optimality conditions. This issue is then often overcome by smoothing the lower level problem. A different approach is presented in \cite{ochs2015bilevel}, where instead of smoothing the lower level problem,  it is suggested to replace the lower level problem constraint by a differentiable update rule, which is given as the n-th step in an iterative procedure to determine approximate solutions to the lower level problem. A bilevel optimization approach to choosing regularization operators for which no ground truth training data is needed is considered in \cite{Hintermueller2017}.

\section{Nonlocal energy spaces}\label{sec:preliminaries}
\subsection{Preliminaries}
Throughout this work, unless otherwise stated,  we let $s \in (0,1)$ and let $\Omega$ denote a nonempty, open, connected, and bounded Lipschitz domain in $\R^N$, where $N \in \N$. Furthermore, $| \cdot |$ denotes the Euclidian norm of a vector in $\R^N$. 
Following \cite[Section 4]{Du2012}, we introduce the notion of a nonlocal energy seminorm. 
\begin{definition}[nonlocal energy]
Let $\gamma \in L^\infty(\Omega \times \Omega)$ be nonnegative and symmetric \ale on $\Omega \times \Omega$. For $u \in \Ltwo$ define a \emph{nonlocal energy seminorm} by
\begin{equation*}
| u |_{\gamma,s} \coloneqq \left(\, \,  \iint_{\Omega \times \Omega }  \frac{|u(y) - u(x)|^2}{|x-y|^{N+2s}} \gamma(x,y) \, \mathrm{d}y \mathrm{d}x\right)^{1/2}.
\end{equation*}
The corresponding \emph{nonlocal energy space} is defined by
\begin{equation*}
V^{\gamma,s}(\Omega) \coloneqq \left\{ v \in L^2(\Omega) : | v |_{\gamma,s} < \infty \right\}
\end{equation*}
and endowed with the norm
\begin{equation}\label{eq:nonlocal_energy_norm}
\|u\|_{V^{\gamma,s}(\Omega) } \coloneqq \left(\|u\|_{L^2(\Omega)}^2 + |u|_{\gamma,s}^2 \right)^{1/2}.
\end{equation}
\end{definition}
\begin{remark}
If $\gamma$ is equal to $1$ almost everywhere on $\Omega \times \Omega$, we write 
\begin{equation*}
\Hs \coloneqq V^{\gamma,s}(\Omega), \quad |u|_{\Hs} \coloneqq |u|_{\gamma,s}, \quad \text{and} \quad \| u \|_{\Hs} \coloneqq \|u \|_{V^{\gamma,s}(\Omega)}.
\end{equation*}
With this notation, $\Hs$ coincides with the usual Sobolev space of fractional order $s$ (also known as Sobolev-Slobodeckij space), see \eg \cite{Nezza2012} and \cite[Definition 1.3.2.1]{Grisvard1986}. 

\end{remark}

We now provide a set of assumptions on the weight $\gamma$, under which the nonlocal energy norm $\|\cdot \|_{V^{\gamma,s}(\Omega)}$ defined by \eqref{eq:nonlocal_energy_norm} is equivalent to the fractional order Sobolev norm $\|\cdot \|_{\Hs}$, which in turn implies that the corresponding nonlocal energy space $V^{\gamma,s}(\Omega)$ coincides with the fractional order Sobolev space $\Hs$.

\begin{assumption}\label{ass:kernelass1}
The weight $\gamma \in L^\infty(\Omega \times \Omega)$ is nonnegative and symmetric \ale on $\Omega \times \Omega$. Furthermore, there exist constants $\gamma_1, \gamma_2, \delta >0$ such that for almost all $(x,y) \in \Omega \times \Omega$ the following statements hold:
\begin{enumerate}[label=\roman*)]
\item If $| x- y | \leq \delta$, then $\gamma_1 \leq \gamma(x,y)$.
\item $\gamma(x,y) \leq \gamma_2$.
\end{enumerate}
\end{assumption}
\begin{remark} If, given any even function $\kappa \colon B_\delta(0) \to  \R$ satisfying $
\gamma_1 \leq \kappa(z) \leq \gamma_2$ for all $z \in B_\delta(0)$, we define $\gamma$ by
\begin{equation*}
\gamma(x,y) \coloneqq \begin{cases}
\kappa(x-y) \quad &\text{if} \quad |x-y| \leq \delta, \\
0 \quad &\text{else},
\end{cases} \quad \text{for almost all } (x,y) \in \Omega \times \Omega,
\end{equation*}
then $\gamma$ satisfies \Cref{ass:kernelass1}.
\end{remark}
The following result is a combination of Lemmas 4.1 and 4.2 from \cite{Du2012}.
\begin{lemma}\label{lem:hsnormbounded}
Let $\gamma \in L^\infty(\Omega \times \Omega)$ satisfy \Cref{ass:kernelass1} and let $u \in \Ltwo$. Then
\begin{equation}\label{eq:hs_semi_norm_bounded}
 |u|_{\gamma,s}^2 \leq \gamma_2 |u|^2_{\Hs}  \quad \text{and} \quad |u|^2_{\Hs} \leq \gamma_1^{-1} |u|_{\gamma,s}^2 + 4 | \Omega | \delta^{-N-2s} \|u\|^2_{L^2(\Omega)}.
\end{equation}
Here, $|\Omega|$ denotes the Lebesgue-measure of $\Omega$.
\end{lemma}

As a direct corollary of \Cref{lem:hsnormbounded} we obtain that if $\gamma \in L^\infty(\Omega \times \Omega)$ satisfies \Cref{ass:kernelass1}, then the corresponding nonlocal energy space is topologically equivalent to the fractional order Sobolev space $\Hs$.

\begin{corollary}[Equivalence of norms]\label{cor:equivalenceofnorms}
There exist constants $m,M >0$ such that for all $\gamma \in L^\infty(\Omega \times \Omega)$ satisfying \Cref{ass:kernelass1} we have
\begin{equation*}
m \|u\|_{\Hs} \leq \|u\|_{V^{\gamma,s}(\Omega)} \leq M \| u \|_{\Hs}, \quad \text{for all } u \in \Ltwo.
\end{equation*}
In particular, the norms on $\Hs$ and $V^{\gamma,s}(\Omega)$ are equivalent.
\end{corollary}
It is straightforward to verify that
\begin{equation*}
\langle u, v \rangle_{\gamma,s} \coloneqq + \iint_{\Omega \times \Omega} \frac{\left(u(x) - u(y)\right)\left(v(x)-v(y)\right) }{|x-y|^{N+2s}} \gamma(x,y) \, \dx\dy \quad 
\end{equation*}
defines a symmetric and positive semidefinite bilinear form on $V^{\gamma,s}(\Omega)$. Moreover, if we let $\langle \cdot , \cdot \rangle_{V^{\gamma,s}(\Omega)} \coloneqq \langle \cdot , \cdot \rangle_\Ltwo + \langle \cdot, \cdot \rangle_{\gamma, s}$, then $\langle \cdot, \cdot \rangle_{V^{\gamma,s}(\Omega)}$ is an inner product on $V^{\gamma,s}(\Omega)$ that induces the norm $\| \cdot \|_{V^{\gamma, s}(\Omega)}$. Since $\Hs$ is complete, the equivalence of norms established in \Cref{cor:equivalenceofnorms} now implies that $V^{\gamma,s}(\Omega)$ is also complete. Hence, $V^{\gamma,s}(\Omega)$ is a Hilbert space. We let
\begin{equation*}
\Pi^0(\Omega) \coloneqq \left\{ f \in L^2(\Omega) \mid \exists c \in \R \colon f(x) = c \quad \text{a.e. on } \Omega \right\}
\end{equation*}
denote the space of functions in $L^2(\Omega)$ which are constant \ale on $\Omega$. We denote by $Q^0 \colon \Ltwo \to \Pi^0(\Omega)$ the $L^2$-orthogonal projection on $\Pi^0(\Omega)$. For $u \in \Ltwo$ we have
\begin{equation*}
Q^0 u (x) = c, \quad \text{ where } c \coloneqq (1/|\Omega|) \int_{\Omega} u(x) \mathrm{d} x.
\end{equation*}
\begin{lemma}\label{lem:kernel_of_seminorm}
Let $\gamma \in L^\infty(\Omega \times \Omega)$ satisfy \Cref{ass:kernelass1}. Then for $u \in \Ltwo$ we have $|u|_{\gamma,s}=0$ if and only if $u \in \Pi^0(\Omega)$.
\end{lemma}
\begin{proof}
Using Fubini's theorem, it is straightforward to verify that for $u \in \Pi^0(\Omega)$ we have $|u|_{\gamma,s} = 0$. Conversely, if for $u \in \Hs$ it holds that $|u|_{\gamma,s}=0$, then
\begin{equation*}
\int_\Omega \frac{|u(x) - u(y)|^2}{|x-y|^{N+2s}} \gamma(x,y) \, \dy = 0 \quad \quad \text{for almost all } x \in \Omega.
\end{equation*}
This implies that $u$ is constant \ale on $B_\delta(x) \cap \Omega$ for almost all $x \in \Omega$. Since $\Omega$ is connected, the claim follows by standard arguments. 
\end{proof}
\begin{remark}
The requirement that $\Omega$ is connected is essential to ensure that for every $\sigma$ satisfying \Cref{ass:kernelass1} and every $u \in \Ltwo$ the seminorm $|u|_{\sigma,s}$ is zero if and only if $u$ is constant almost everywhere. The fractional order Sobolev seminorm $|u|_{H^s(\Omega}$, however, has this property for all open sets $\Omega$, connected or not. The reason for this is that while the weight is equal to 1 almost everywhere for the fractional order Sobolev norm, in general $\sigma(|x-y|)$ might be zero for $| x - y | > \delta$.
\end{remark}

\begin{lemma}[Poincare-Wirtinger inequality for nonlocal energy spaces]\label{lem:wirtinger}
There exists a constant $C > 0$ such that for every $\gamma \in L^\infty(\Omega \times \Omega)$ satisfying \Cref{ass:kernelass1} and every $u \in \Hs$ with $Q^0 u = 0$ it holds that
\begin{equation}\label{eq:poincare_wirtinger}
\|u \|_{\Ltwo} \leq C |u |_{\gamma, s}.
\end{equation}
\end{lemma}
\begin{proof}
Let $\gamma_{\min} \coloneqq \gamma_1 \chi_{A}$, where $\chi_A$ denotes the characteristic function of the set $A \coloneqq \{(x,y) \in \Omega \times \Omega \mid |x-y|\geq \delta \}$. It is easy to show that $|u|_{\gamma_{\min},s} \leq |u|_{\gamma,s} $ for every $u \in \Hs$ and every $\gamma$ satisfying \Cref{ass:kernelass1}. Consequently, it suffices to prove the claim for $\gamma = \gamma_{\min}$. We argue by contradiction. If the claim is wrong, then there is a sequence $(\un)$ in $\Hs$ such that 
\begin{equation*}
Q^0 \un = 0, \quad \|\un\|_{\Ltwo}=1, \quad \text{and} \quad |\un|_{\gamma_{\min},s} \leq 1/n \quad \text{for all } n \in \N.
\end{equation*}
Using the equivalence of norms established in \Cref{cor:equivalenceofnorms}, it follows that $(\un)$ is bounded in $\Hs$. Since $\Hs$ is reflexive, $(u^n)$ has an accumulation point $u \in \Hs$ with respect to the weak topology on $\Hs$. Since $\Hs$ is compactly embedded in $\Ltwo$ (see \cite[Theorem 1.4.3.2]{Grisvard1986}), it follows that $u$ is also an accumulation point of $(u^n)$ with respect to the strong topology on $\Ltwo$.
Using the continuity of $Q^0$ on $\Ltwo$ and the weak lower semi continuity of the nonlocal energy seminorm we deduce that
$Q^0 u =0$ and $|u|_{\gamma_{\min},s} = 0$.
From \Cref{lem:kernel_of_seminorm} it is clear that this implies $u=0$.
However, since $u$ is an accumulation point of $(u^n)$ with respect to the strong topology on $\Ltwo$, we must also have $\| u \|_{\Ltwo} = \lim_{n \to \infty} \| u^n  \|_\Ltwo = 1$,
which is a contradiction. Hence, the proof is finished.
\end{proof}

\subsection{A stability property}
\begin{lemma}[Stability]\label{lem:stability}
Let $1 \leq p , q \leq \infty$ be such that $1/p+1/q = 1$ and $\Hs$ is compactly embedded in $\Lq$. Let $\un \rightharpoonup u$ in $\Hs$, $\gamma^n \rightharpoonup^* \gamma$ in $L^\infty(\Omega \times \Omega)$, and  $\pn \rightharpoonup p \in \Lp$, where $(\un)$ in $\Hs$, $(\gamma^n)$ in $L^\infty(\Omega \times \Omega)$, and $(\pn)$ in $\Lp$ are sequences related by

\begin{equation}\label{eq:stability_assumption}
\langle u^n, v \rangle_{\gamma^n,s} = \langle p^n, v \rangle_{\Lp, \Lq}  \quad  \text{for all } v \in \Hs \text{ and all } n \in \N.
\end{equation}
Then 
\begin{equation}\label{eq:1}
\langle u, v \rangle_{\gamma,s} = \langle p, v \rangle_{\Lp,\Lq} \quad \text{for all } v \in \Hs \quad \text{and} \quad \lim_{n \to \infty} |u^n|_{\gamma^n,s}^2 = |u|_{\gamma,s}^2 .
\end{equation}

\end{lemma}
\begin{proof}
The proof is divided into three steps. 
\begin{description}

\item [Step 1:] We show that for all $v \in C_c^\infty(\bar{\Omega})$
\begin{equation*}
\frac{(\un(x)- \un(y))(v(x)-v(y))}{|x-y|^{N+2s}} \to \frac{(u(x)- u(y))(v(x)-v(y))}{|x-y|^{N+2s}} \quad \text{in } L^1(\Omega \times \Omega).
\end{equation*}
Since $s<1$ we can find $\varepsilon > 0$ such that $s + \varepsilon < 1$ and $s' \coloneqq s - \varepsilon > 0$.
Since $s' < s$, we have $H^{s'}(\Omega)$ is compactly embedded in $\Hs$ (see \cite[Theorem 1.4.3.2]{Grisvard1986}) and thus
$\un \to u$ in $H^{s'}(\Omega)$. Using H\"older's inequality for the first, and the mean value theorem for the second inequality below, we estimate
\begin{multline*}
\iint_{\Omega \times \Omega} \frac{|(\un(x) -\un(y)) -(u(x) - u(y))| |v(x) - v(y)|}{|x-y|^{N+2s}} \dx \dy \\  \leq | \un - u |_{H^{s'}(\Omega)} \left( \, \, \iint_{\Omega \times \Omega} \frac{(v(x) - v(y))^2}{|x-y|^{N+2(2s-s')}} \dx \dy \right)^{1/2} \\ \leq | \un  - u |_{H^{s'}(\Omega)} \| D v\|_\Linfty \underbrace{\left( \, \,\iint_{\Omega \times \Omega} \frac{1}{|x-y|^{N+2(2s-s'-1)}} \dx \dy \right)^{1/2}}_{C},
\end{multline*}
where $C < \infty$ since $2s - s' - 1  = s +\varepsilon -1 < 0$.

\item [Step 2:] We compute
\begin{multline}\label{eq:eqtest}
\langle p, v \rangle_{\Lp,\Lq} = \lim\limits_{n \to \infty} \langle \pn, v\rangle_{\Lp,\Lq} \\ = \lim\limits_{n \to \infty} \iint_{\Omega \times \Omega} \frac{(\un(x) - \un(y))(v(x)-v(y)) \gamma^n(x,y)}{|x-y|^{N+2s}} \dx \dy \\ = \iint_{\Omega \times \Omega} \frac{(u(x) - u(y))(v(x)-v(y)) \gamma(x,y)}{|x-y|^{N+2s}} \dx \dy \quad \text{for all } v \in C_c^\infty(\bar{\Omega}).
\end{multline}
where the result from the first step justifies the third equality. Recalling that $C_c^\infty(\bar{\Omega})$ is dense in $\Hs$ (see \cite[Theorem 1.4.2.1]{Grisvard1986}) and observing that \eqref{eq:eqtest} is continuous with respect to $v$ on $\Hs$, it follows that \eqref{eq:eqtest} holds for all $v \in \Hs$. This proves the first equality in \eqref{eq:1}.

\item [Step 3:] It remains to prove the second equality in \eqref{eq:1}. Subtracting the first equality in \eqref{eq:1} with $v = u$ from \eqref{eq:stability_assumption} with $v = u^n$, we obtain
\begin{equation}\label{eq:limit_testing_stability}
|u^n|_{\gamma^n,s}^2 - |u|_{\gamma,s}^2 = \langle \pn, \un \rangle_{\Lp,\Lq} - \langle p, u \rangle_{\Lp,\Lq}
\end{equation}
Since $\un \to u$ in $\Lq$ and $p^n \rightharpoonup p$ in $\Lp$ it follows that the right hand side of \eqref{eq:limit_testing_stability} tends to zero as $n \to \infty$. This finishes the proof. 
\end{description}
\end{proof}

\subsection{Distance dependent weights}
Let $d$ denote the diameter of $\Omega$, \ie $\diam \coloneqq \sup_{x,y \in \Omega} | x-y|$. From now on we restrict ourselves to $\gamma \in L^\infty(\Omega \times \Omega)$ of the form
\begin{equation}\label{eq:extension}
\gamma(x,y) = \sigma(|x-y|) \quad \text{\ale on }  \Omega \times \Omega,
\end{equation}
where $\sigma \in L^\infty((0,d))$ satisfies the following conditions:
\begin{enumerate}[label= (A\arabic*)]
\item $0 \leq \sigma(t) \leq \gamma_2$ \ale on $(0,d)$, \label{item:weight_upper_bound}
\item $0 < \gamma_1 \leq \sigma(t)$ \ale on $(0,\delta)$. \label{item:weight_lower_bound}
\end{enumerate}
To simplify notation, if $\sigma$ and $\gamma$ are related by \eqref{eq:extension}, then we write $|u|_{\sigma,s} \coloneqq |u|_{\gamma, s}$.
The set of feasible weights is defined by
\begin{equation*}
\Wad \coloneqq \Big\{ \sigma \in L^\infty((0,d)) \mid \sigma \text{ satisfies } \ref{item:weight_upper_bound} \text{ and } \ref{item:weight_lower_bound} \Big\},
\end{equation*}
Some care must be taken, since it is not immediately clear, although intuitively reasonable, that for every $\sigma \in L^\infty((0,d))$ there is $\gamma \in L^\infty(\Omega \times \Omega)$ satisfying \eqref{eq:extension}. The difficulty stems from the fact that $L^\infty((0,d))$ consists only of equivalence classes of functions coinciding in the almost everywhere sense on $(0,d)$. We emphasize that \eqref{eq:extension} must be understood in the sense that it holds for all representatives of the equivalence classes $\sigma$ and $\gamma$. In the following proposition we confirm that the assumption that $\gamma \in L^\infty(\Omega \times \Omega)$ is well-defined by \eqref{eq:extension} for any $\sigma \in L^\infty((0,d))$ is indeed justified. To avoid confusion between equivalence classes of functions and their representatives, in the following proposition, we use the special notation $\lfloor \cdot \rfloor$ to denote equivalence classes of functions.

\begin{proposition}\label{prop:embedding_in_Linfty}
For every $\lfloor \phi \rfloor \in L^\infty((0,\diam))$ there exists a unique $\lfloor \gamma \rfloor \in L^\infty(\Omega \times \Omega)$ such that for all $\phi \in \lfloor \phi \rfloor$ and $\gamma \in \lfloor \gamma \rfloor$
\begin{equation}\label{eq:extension2}
\gamma(x,y) = \phi(| x - y | ) \quad \text{\ale on } \Omega \times \Omega.
\end{equation}
Moreover, it holds that
\begin{equation}\label{eq:continuous_embedding}
\| \lfloor \gamma \rfloor \|_{L^\infty(\Omega \times \Omega)} \leq \| \lfloor \phi \rfloor \|_{L^\infty((0,\diam))}.
\end{equation}
\end{proposition}
\begin{proof}
First, we take a particular representative $\phi \colon (0,d) \to \R$ of the equivalence class of measurable functions $\lfloor \phi \rfloor \in L^\infty((0,d))$ to define the equivalence class of functions $\lfloor \gamma \rfloor$ as the set of all measurable functions $\gamma : \Omega \times \Omega \to \R$ satisfying 
\begin{equation*}
\gamma(x,y) = \phi(| x - y| ) \quad \text{for almost all } (x,y) \in \Omega \times \Omega.
\end{equation*}
We now prove that $\lfloor \gamma \psi \rfloor \in L^1(\Omega \times \Omega)$ for every $\lfloor \psi \rfloor \in L^1(\Omega \times \Omega)$. Using that Fubini's theorem and polar coordinates can be employed for all nonnegative and measurable functions (see \cite[Theorems 2.39 and 2.49]{Folland1999realanalysis}), this follows from the estimate
\begin{multline*}
\|\lfloor \gamma \psi \rfloor \|_{L^1(\Omega \times \Omega)} = \int_\Omega \int_0^{\diam} \int_{\partial B(x,r) \cap \Omega} | \phi(|x-y|) \psi(x,y)| \, \mathrm{d} S(y) \mathrm{d}r   \dx \\ = \int_\Omega \int_0^{\diam}  |\phi(r)| \int_{\partial B(x,r) \cap \Omega }  |\psi(x,y)| \, \mathrm{d} S(y) \mathrm{d}r  \dx\leq \| \lfloor \phi \rfloor \|_{L^\infty((0,\diam))} \| \lfloor \psi \rfloor \|_{L^1(\Omega \times \Omega)}.
\end{multline*}
Here and in the sequel $S$ denotes the $N-1$ dimensional surface measure on a sphere with radius $r$.
Having established that $\lfloor \gamma \psi \rfloor \in L^1(\Omega \times \Omega)$, we use Fubini's theorem and polar coordinates for real valued integrable functions to obtain
\begin{equation*}
\iint_{\Omega \times \Omega} \gamma(x,y) \psi(x,y) \dx dy = \int_\Omega \int_0^{\diam} \phi(r) \int_{\partial B(x,r) \cap \Omega}  \psi(x,y) \, \mathrm{d} S(y) \mathrm{d}r  \dx.
\end{equation*}
Note that the right-hand side of this equation is independent of the particular representative of $\lfloor \phi \rfloor$ used to define $\lfloor \gamma \rfloor$. Thus, since $\lfloor \psi \rfloor \in L^1(\Omega \times \Omega) $ was arbitrary, by the fundamental lemma of calculus of variations, the definition of $\lfloor\gamma \rfloor $ is independent of the particular representative of $\lfloor \phi \rfloor$ chosen to define $\lfloor \gamma \rfloor$. This implies that $\lfloor \gamma \rfloor $ is well-defined by \eqref{eq:extension2} as an equivalence class of functions on $\Omega \times \Omega$. It remains to show \eqref{eq:continuous_embedding}. To do this, we argue as follows: If \eqref{eq:continuous_embedding} does not hold, then there exists $\varepsilon > 0$ and $A \subset \Omega \times \Omega$ with $|A| > 0 $ such that for all $\gamma \in \lfloor \gamma \rfloor$
\begin{equation*}
| \gamma(x,y)| \geq \| \lfloor \phi \rfloor \|_{L^\infty((0,\diam))}+\varepsilon \quad \text{\ale on } A,
\end{equation*}
where $|A|$ denotes the $2N$-dimensional Lebesgue measure of $A$. It follows that
\begin{multline*}
(\|\lfloor  \phi  \rfloor \|_{L^\infty((0,\diam))}+\varepsilon) |A| \\ \leq \iint_{\Omega \times \Omega} |\gamma(x,y)| \chi_A(x,y) \dx\dy =   \int_\Omega \int_0^{\diam} |\phi(r)| \int_{\partial B(x,r) \cap \Omega}  \chi_A(x,y) \, \mathrm{d} S(y) \mathrm{d}r   \dx \\ \leq  \| \phi \|_{L^\infty((0,\diam))} \int_\Omega \int_0^{\diam}  \int_{\partial B(x,r) \cap \Omega}  \chi_A(x,y) \, \mathrm{d} S(y) \mathrm{d}r  \dx \leq \| \lfloor \phi \rfloor \|_{L^\infty((0,\diam))} |A|.
\end{multline*}
However, this can only be true if $|A| = 0$, which contradicts our assumption. 
\end{proof}

\begin{remark}
Note that in the proof of \Cref{prop:embedding_in_Linfty} we use a representation in polar coordinates. Representation in polar coordinates is often seen as a special case of the coarea formula. Unfortunately, the coarea formula as given \eg in \cite{LawrenceCraigEvans1991} can not be directly applied to $\gamma \psi$, since it has a requirement that $\gamma \psi \in L^1(\Omega \times \Omega)$, but this is exactly what we are proving in the first part of the proof. 
\end{remark}
It follows from similar arguments as in \Cref{prop:embedding_in_Linfty}, that $\gamma \in L^\infty(\Omega \times \Omega)$ defined as in \eqref{eq:extension2} satisfies \Cref{ass:kernelass1} for all $\sigma \in \Wad$. Thus, using \Cref{cor:equivalenceofnorms}, there exist constants $m,M >0$ such that for all $\sigma \in \Wad$ it holds that
\begin{equation*}
m \| u \|_{\Hs} \leq \| u \|_{V^{\sigma,s}(\Omega)} \leq M \| u \|_{\Hs}.
\end{equation*}

\begin{lemma}\label{lem:weakstarcontinuity}
The mapping 
\begin{align*}
\Phi \colon L^\infty((0,d)) &\to L^\infty(\Omega \times \Omega) \\
\phi &\mapsto \gamma, \text{ such that }\gamma(x,y)  = \phi(| y - x |) \text{ a.e. on }\Omega \times \Omega,
\end{align*}
is well-defined, linear, continuous, and  sequentially weak$^*$-to-weak$^*$ continuous.
\end{lemma}
\begin{proof}
We have already seen in \Cref{prop:embedding_in_Linfty} that $\Phi$ is well-defined and continuous. Clearly, $\Phi$ is also linear. It remains to show weak$^*$-to-weak$^*$ sequential continuity. To do this, we let $(\phi^n)$ be a sequence in $L^\infty((0,\diam))$ and $\phi \in L^\infty((0,\diam))$ be such that
\begin{equation*}
\phi^n \rightharpoonup^* \phi \quad \text{in } L^\infty((0,\diam)). 
\end{equation*}
Let  $\gamma \in L^\infty(\Omega \times \Omega)$ and the sequence $(\gamma^n)$ in $L^\infty(\Omega \times \Omega)$ be defined by $\gamma(x,y) \coloneqq \phi(| x- y|)$ and $  \gamma^n(x,y) \coloneqq \phi^n(|x-y|)$, respectively, for almost all $(x,y) \in \Omega \times \Omega$ and all $n \in \N$. We must show that $\gamma^n \rightharpoonup^* \gamma \quad \text{in } L^\infty(\Omega \times \Omega)$. For this purpose, let $\psi \in L^1(\Omega \times \Omega)$ be arbitrary. Using \Cref{prop:embedding_in_Linfty} and Fubini's theorem, the integrals
\begin{equation*}
\int_\Omega |\gamma^n(x,y)| |\psi(x,y)| \dy \quad \text{and} \quad \int_\Omega |\psi(x,y)| \, \dy  
\end{equation*}
exist and are finite for almost all $x \in \Omega$ for all $n \in \N$ . Employing polar coordinates and Fubini's theorem, we obtain that for almost all $x \in \Omega$
\begin{multline*}
 \int_\Omega \gamma^n(x,y) \psi(x,y)\, \dy = \int_0^{\diam} \int_{\partial B(x,s) \cap \Omega }   \gamma^n(x,y) \psi(x,y) \mathrm{d} S(y) \, \mathrm{d} s \\ =\int_0^{\diam}  \phi^n(s) \int_{\partial B(x,s) \cap \Omega }   \psi(x,y) \mathrm{d}S(y) \, \mathrm{d} s,
\end{multline*}
and that the map
$
s \mapsto \int_{\partial B(x,s) \cap \Omega } \psi(x,y) \, \mathrm{d}S(y) 
$
 is in $L^1((0,\diam))$. Using that $\phi^n \rightharpoonup^* \phi$ in $L^\infty((0,\diam))$, we deduce that
\begin{equation}\label{eq:pointwiseconvergence}
 \lim_{n \to \infty} \int_\Omega \gamma^n(x,y) \psi(x,y) \dy \dx = \int_\Omega \gamma(x,y) \psi(x,y) \, \dy \dx
\end{equation}
for almost all $x \in \Omega$. We now define
\begin{equation*}
f^n(x) \coloneqq \int_\Omega \gamma^n(x,y) \psi(x,y) \, \dy \quad \text{and} \quad f(x) \coloneqq \int_\Omega \gamma(x,y) \psi(x,y)  \dy . 
\end{equation*}
\Cref{eq:pointwiseconvergence} shows that $f^n(x) \to f(x)$ \ale on $\Omega$. Since moreover
\begin{equation*}
|f^n (x)| \leq \|\phi^n \|_{L^\infty((0,\diam))}  \int_{\Omega} |\psi(x,y) |\, dx \dy \leq C  \int_{\Omega} |\psi(x,y) |\, dx \dy  \quad 
\end{equation*}
and $x \mapsto \int_\Omega | \psi(x,y) | \dy$ is integrable, Lebesgue's dominated convergence theorem asserts that
\begin{equation*}
\lim_{n \to \infty} \iint_{\Omega \times \Omega} \gamma^n (x,y) \psi(x,y) \, \dx \dy = \iint_{\Omega \times \Omega} \gamma (x,y) \psi(x,y) \, \dx \dy.
\end{equation*} 
Since $\psi \in L^1(\Omega \times \Omega)$ was chosen arbitrarily, this is precisely what we needed to show.
\end{proof}
\begin{remark} It can be of interest to take a note of the interpretation of nonlocal energy seminorms for $\Omega = \R$ using Fourier analysis. Here, $\mathscr{F} \colon L^2(\R) \to L^2(\R)$ denotes the usual Fourier transform, see \eg \cite[Section 4.12] {BrediesLorenz2010_Math}. Similarly as in \cite[Proposition 3.4]{Nezza2012}, for $u \in C_c(\R)$ we compute
\begin{multline*}
| u |_{\sigma,s}^2 {=}  \int_\R \left\| \mathscr{F}\left(\frac{|u(z+\cdot) - u(\cdot)|}{|z|^{1/2+s}} \sigma(|z|)^{1/2} \right) \right\|_{L^2(\R)}^2 \dz   \\= \int_\R \left(\int_\R \frac{|e^{i \xi z} - 1|^2}{|z|^{1+2s}} \sigma(|z|) |\mathscr{F}(u(\xi))|^2 \dxi \right) \dz  \\ \stackrel{v=\xi z}{=} \int_\R \left(\int_\R \frac{|e^{i v} - 1|^2}{|v|^{1+2s}} \sigma(|v/\xi|) |\xi|^{2s}|\mathscr{F}(u(\xi))|^2 \dxi \right) \dv.
\end{multline*}
Here we use Plancherel's theorem (see \cite[Satz 4.16]{BrediesLorenz2010_Math}) and the fact that translation in the time domain corresponds to modulation in the frequency domain (see \cite[Lemma 4.5]{BrediesLorenz2010_Math}) for the first and second equality, respectively. If we let
\begin{equation*}
\rho(\xi) \coloneqq \int_\R  \frac{|e^{iv}-1|^2}{|v|^{1+2s}} \sigma(v/\xi)) \dv,
\end{equation*}
then it follows that
\begin{equation*}
|u|_{\sigma,s}^2 = \int_\R \rho(\xi) |\xi|^{2s} |\mathscr{F}(u(\xi))|^2  \dxi.
\end{equation*}
This shows that nonlocal energy seminorms behave similar as the fractional order Sobolev-Slobodeckij seminorm (see \cite{Nezza2012}). The only difference is an additional weighting term $\rho$ depending on the frequency.
\end{remark}

\section{Linear case}
\subsection{Problem setting}
For $\sigma \in \Wad$, we consider the lower level problems
\begin{equation}\label{prob:optcontrol1}\tag{\mbox{$P(\sigma)$}}
\min_{ u \in  \Hs} \| S u - \ydelta \|_{\Ltwo}^2 + |u|_{\sigma, s}^2, 
\end{equation}
where $S \in \Lin(L^2(\Omega))$, and $\ydelta \in \Ltwo$ is a noisy measurement of the ground truth state.  The solution set of the lower level problems is denoted by
\begin{equation*}
\m{F} \coloneqq \left\{ (\sigma, u) \in \Wad \times \Hs \mid u \text{ solves } \eqref{prob:optcontrol1} \right\}.
\end{equation*}
We address the following learning problem
\begin{equation}\label{prob:bilevelsimple}\tag{BP}
\min_{\sigma \in \Wad, u \in H^{s}(\Omega)} \frac{1}{2} \| u - \uex \|_{\Ltwo}^2 + R(\sigma) \quad  \text{subject to} \quad (\sigma,u) \in \m{F},
\end{equation}
where $\uex \in \Ltwo$ is the ground truth control and $R \colon L^\infty((0,\diam)) \to [0,\infty)$ is a given regularization operator. We emphasize that when showing existence of solutions we specifically include the case $R \equiv 0$, \ie we do not require additional regularization of the weights.

\begin{example}\label{ex:poissoneq}
As an example let $S$ be the solution operator to an elliptic PDE with Neumann boundary conditions. More precisely, for $(y,u) \in \Ltwo \times \Ltwo$ let
\begin{equation*}
S u = y \quad \text{if and only if} \quad -\rho \Delta y +y = u \text{ and } \frac{\partial y }{\partial n}=0 \text{ on } \partial \Omega,
\end{equation*}
where $\rho$ is a positive constant. Clearly, $\rg{(S)} \subset \He$. As a straightforward consequence of the Lax-Milgram lemma and the standard Sobolev embedding, it follows that $S$ is a compact operator on $\Ltwo$. Moreover, it is easy to see that $S$ is injective. 
\end{example}

\subsection{Preliminaries}

\begin{proposition}[Uniform convexity]\label{prop:coercivity}
Let $S$ be injective on $\Pi_0(\Omega)$. Then there exist $c,C > 0$ such that for every $\sigma \in \Wad$  and every $u \in \Hs$ we have
\begin{equation}
c \| u \|_{\Hs}^2 \leq \| S u \|_{\Ltwo}^2 + | u |_{\sigma,s}^2 \leq C \| u \|_{\Hs}^2.
\end{equation}
\end{proposition}
\begin{proof}
The second inequality follows from the continuity of $S$ on $\Ltwo$ and \Cref{cor:equivalenceofnorms}. To prove the first inequality, first note that for all $\sigma \in \Wad$ it holds that
\begin{equation*}
| u |_{\sigma_{\min},s} \leq | u |_{\sigma,s} \quad \text{for all } u \in \Hs,
\end{equation*}
where $\sigma_{\min} \coloneqq \gamma_1 \chi_{(0,\delta)}$. Thus, it suffices to prove the first inequality for $\sigma = \sigma_{\min}$. We begin by showing that there exists $c_1 > 0$ such that
\begin{equation}\label{eq:L2normbound}
c_1 \| u \|_{\Ltwo}^2 \leq \| S u \|_{\Ltwo}^2 + | u |_{\sigma_{\min},s}^2 \quad \text{for all } u \in \Hs.
\end{equation}
To do this, we argue by contradiction. If there is no $c_1>$0 such that \eqref{eq:L2normbound} holds, then there exists a sequence $(\un)$ in $\Hs$ such that
\begin{equation*}
\| \un \|_\Ltwo = 1,  \quad \text{and} \quad \| S\un \|_\Ltwo^2 + |\un |_{\sigma_{\min},s}^2 \leq 1/n \quad \text{ for all } n \in \N.
\end{equation*}
Using \Cref{cor:equivalenceofnorms}, we deduce that $(\un)$ is bounded in $\Hs$. Since $\Hs$ is reflexive, it follows that $(u^n)$ has an accumulation point with respect to the weak topology on $\Hs$. Moreover, $\Hs$ is compactly embedded in $\Ltwo$, and consequently it follows that $u$ is also an accumulation point of $(u^n)$ with respect to the strong topology on $\Ltwo$. Standard arguments show that 
\begin{equation*}
\|u\|_\Ltwo = 1, \quad u \in \ker | \cdot |_{\sigma_{\min},s} = \Pi_0(\Omega) \quad \text{and} \quad S u = 0,
\end{equation*}
This contradicts the assumption that $S$ is injective on $\Pi_0(\Omega)$. Hence we have proven that there is $c_1>0$ such that \eqref{eq:L2normbound} holds. Since by \Cref{lem:hsnormbounded} we already know that
\begin{equation*}
|u|^2_{\Hs} \leq \gamma_1 |u|_{\sigma_{\min},s}^2 + 4 | \Omega | \delta^{-N-2s} \|u\|^2_{\Ltwo},
\end{equation*}
the claim follows by straightforward computations.
\end{proof}

\begin{proposition}\label{prop:exsolllop}
Assume that $S$ is injective on $\Pi_0(\Omega)$. Then \eqref{prob:optcontrol1} has a unique solution. 
\end{proposition}
\begin{proof}
Since \eqref{prob:optcontrol1} is a convex minimization problem, it follows that $u^* \in \Hs$ solves \eqref{prob:optcontrol1} if and only if $u^*$ satisfies the first order optimality condition 
\begin{equation}\label{eq:optcondLLPprop}
\langle S u^*, S v \rangle_\Ltwo  + \langle u^*, v \rangle_{\sigma,s} = \langle \ydelta, S v  \rangle_\Ltwo \quad \text{for all } v \in \Hs.
\end{equation}
Existence and uniqueness of solutions to \eqref{eq:optcondLLPprop} can be easily proven using the Lax-Milgram lemma (the required coercivity is a direct consequence of  
\Cref{prop:coercivity}).
\end{proof}

\paragraph{Optimality conditions for the lower level problem}
To simplify notation, we define $\m{L}^s(\sigma) \colon \Hs \to \Hs\dual$  by
\begin{multline*}
[\m{L}^s(\sigma) u] v = \iint_{\Omega \times \Omega} \frac{(u(x)-u(y))(v(x)-v(y))}{|x-y|^{N+2s}} \sigma(|x-y|) \, \dx\dy \\ \quad \text{for } (u,v) \in \Hs \times \Hs.
\end{multline*}
Note that $\m{L}^s(\sigma)$ is a bounded operator from $\Hs$ to $\Hs\dual$. Since \eqref{prob:optcontrol1} is convex, we obtain the following necessary and sufficient optimality condition for \eqref{prob:optcontrol1}.

\begin{proposition}\label{prop:opt_sol_llp}
An element  $u \in \Hs$ solves \eqref{prob:optcontrol1} if and only if 
\begin{equation*}
S^* S u -S^* \ydelta+ \m{L}^s(\sigma) u = 0  \quad \text{in } \Hs\dual.
\end{equation*}
\end{proposition}

\begin{remark}
An optimality system for \eqref{prob:optcontrol1} with  $S$ as in \Cref{ex:poissoneq} can be obtained by standard Lagrangian methods. Indeed, let $u \in \Hs$ be a solution to \eqref{prob:optcontrol1} with $S$ as in \Cref{ex:poissoneq}. Then there exists $p \in \Hen$ such that
\begin{subequations}
\begin{align}
[\m{L}^s(\sigma)u] w + \int_\Omega p(x) w(x) \, \dx &= 0, \quad \text{for all }  w  \in \Hs \tag{optimality},\\
\int_{\Omega} [y(x)-\ydelta(x)]v(x) + \grad p(x) \cdot\grad v(x) \, \dx &= 0, \quad \text{for all }  v  \in \Hen, \tag{adjoint eq.}\\
\int_\Omega \grad y(x) \cdot \grad z(x) - u(x) z(x) \, \dx &=0, \quad \text{for all } z  \in \Hen. \tag{state eq.}
\end{align}
\end{subequations}
\end{remark}

\subsection{Existence of solutions}

To prove that \eqref{prob:bilevelsimple} has a solution, we apply the direct method of the calculus of variations. The crucial step in the proof is the argument proving that the feasible set is sequentially closed with respect to weak$^*$ convergence. Since the feasible set is defined by the lower level problem, this is related to stability of the lower level problem with respect to the weight function.

\begin{proposition}
Assume that $S$ is injective on $\Pi_0(\Omega)$ and that $R$ is weak$^*$ sequentially lower semi-continuous. Then \eqref{prob:bilevelsimple} has a solution.
\end{proposition}
\begin{proof}
First of all note that as a consequence of \Cref{prop:exsolllop} the feasible set $\mF$ is nonempty. Thus, we can take a minimizing sequence for \eqref{prob:bilevelsimple}, \ie a sequence $(\sigman, \un) \in \mF$ such that 
\begin{equation*}
\lim_{n \to \infty} \frac{1}{2} \| \un - \uex \|_{\Ltwo}^2 + R(\sigman) = \inf_{(u,\sigma) \in \Fad}  \frac{1}{2} \|u - \uex\|_{\Ltwo}^2 + R(\sigma).
\end{equation*}
It is easy to prove that $\Wad$ is sequentially weak$^*$ compact. Consequently, $(\sigman)$ has a subsequence, again denoted by $(\sigman)$, such that $\sigman \rightharpoonup^* \sigma$ in $L^\infty((0,d))$ for some $\sigma \in \Wad$. \Cref{lem:weakstarcontinuity} ensures that in this case $\sigma^n(x,y) \coloneqq \sigman(|x-y|)$ converges to $\sigma^*(x,y)\coloneqq \sigma^*(|x-y|)$ with respect to the weak$^*$ topology on $L^\infty(\Omega \times \Omega)$. Using that
\begin{equation*}
\| S \un - \ydelta \|_\Ltwo^2 + | \un |_{\sigma^n,s}^2 \leq   \| \ydelta \|_\Ltwo \quad \text{for all } n \in \N,
\end{equation*}
it follows from \Cref{prop:coercivity} that $(\un)$ is bounded in $H^{s}(\Omega)$. Since $\Hs$ is reflexive, this implies that $(\un)$ has a subsequence, which we again denote by $(\un)$, such that $\un \rightharpoonup u^\ast$ in $\Hs$ for some $u^\ast \in H^{s}(\Omega)$. Note that if we set
\begin{equation*}
\pn \coloneqq S^* S \un - S^* \ydelta \quad \text{for all } n \in \N,
\end{equation*}
then $(\pn)$ is a bounded sequence in $\Ltwo$, and thus it has a weakly converging subsequence, which we again denote by $(\pn)$, such that $\pn \rightharpoonup p$ in $\Ltwo$. As a consequence of \Cref{lem:stability} we now obtain
$|u^*|_{\sigma^*,s}^2  = \lim_{n \to \infty} |u^n|_{\sigma^n, s}^2$
From this, and using that $\un$ solves the lower level problem \eqref{prob:optcontrol1}, with $\sigman$ in place of $\sigma$, to justify the second inequality below, we deduce that 
\begin{multline*}
\| Su^* - \ydelta \|_\Ltwo^2 + |u|_{\sigma^*,s}^2 \leq \liminf \limits_{n \to \infty} \| S\un - \ydelta \|_\Ltwo^2 + |\un|_{\sigma^n,s}^2 \\ \leq \lim\limits_{n \to \infty} \| Su - \ydelta \|_2^2 + |u|_{\sigma^n,s}^2 = \| Su - \ydelta \|_\Ltwo^2 + |u|_{\sigma^*,s}^2
\end{multline*}
for all $u \in \Hs$. Note that the last equality follows from the weak$^*$ convergence of $\sigma^*(|x-y|)$. This shows that $(\sigma^\ast,u^*) \in \mF$. Due to the weak$^*$ lower semi continuity of the involved functions we have
\begin{equation*}
\frac{1}{2} \|u^* - \uex \|_{\Ltwo}^2 + R(\sigma^*) \leq \lim\limits_{n \to \infty} \frac{1}{2} \|\un - \uex \|_{\Ltwo}^2 + R(\sigman).
\end{equation*}
Since $(\sigman, \yn, \un)$ was chosen as a minimizing sequence of \eqref{prob:bilevelsimple}, and $(\sigma^*,u^*) \in \mF$, this implies that
\begin{equation*}
\frac{1}{2} \|u^* - \uex \|_{\Ltwo}^2 + R(\sigma^*) = \inf\limits_{(u,\sigma) \in \Fad} \frac{1}{2} \| u - \uex \|_{\Ltwo}^2 + R(\sigma)
\end{equation*}
which shows that $(\sigma^*, u^*)$ is a solution to \eqref{prob:bilevelsimple}.
\end{proof}

\subsection{Optimality conditions}
In the following we derive optimality conditions for the bilevel problem. Recall that the Lagrange function 
$L \colon \Wad \times \Hs \times \Hs \to \R$ of the bilevel problem is such that
\begin{equation*}
(\sigma, u,q) \mapsto \frac{1}{2} \|u - \uex\|_\Ltwo^2 + R(\sigma) +  [\m{L}^s(\sigma) u] q + \langle Su - \ydelta, S q \rangle_\Ltwo.
\end{equation*}

\begin{proposition}[optimality system]
If $(u, \sigma) \in \Fad$ is a solution to \eqref{prob:bilevelsimple}, then there exists $q \in \Hs$ such that
\begin{subequations}
\begin{align}
R'(\sigma) [w - \sigma] + [\m{L}^s(\omega-\sigma) u] q &\geq 0 \quad \text{for all } \omega \in \Wad, \tag{optimality}\\
u - \uex + \m{L}^s(\sigma)q + S^* S q  &= 0  \quad \text{in } \Hs\dual,  \tag{adjoint}\\
S^* S u -S^* \ydelta+ \mathcal{L}^s(\sigma) u &= 0  \quad \text{in } \Hs\dual \tag{constraint}.
\end{align}
\end{subequations}
\end{proposition}
\begin{proof}
Since as a consequence of \Cref{prop:coercivity} the lower level problem is uniformly convex for all $\sigma \in \Wad$, this follows using the standard Lagrangian based approach.
\end{proof}

\section{Nonlinear case}\label{sec:nonlinear_example}
In this section we study how nonlocal regularization operators can be learned for the inverse problem of determining parameters in partial differential equations. In this context, the forward problems is often nonlinear and additional constraints on the parameter set are needed to ensure that the forward problem is well-posed on the feasible set. We only consider scalar pointwise constraints on the parameters. More precisely, for $u_{\min}, u_{\max} \in [-\infty, \infty]$ such that $ u_{\min} < u_{\max}$ we let the feasible set be given by
\begin{equation*}
\Uad \coloneqq \left\{ u \in \Hs \mid u_{\min} \leq u(x) \leq u_{\max} \text{ \ale on } \Omega \right\}.
\end{equation*}
We are interested in the case where the forward operator is only defined implicitly as the solution operator to a PDE which depends on the sought-after parameters. We let the function describing the PDE be denoted by $e \colon Y \times \Uad \to Z$. Here $Z$ is a general Hilbert space, and $Y$ is a Hilbert space that is continuously embedded in $\Ltwo$. If for every $u \in \Uad$ there exists a unique $y \in Y$ such that $e(y,u) = 0$, then the corresponding forward operator $S \colon \Uad \to Y$  is such that $u \mapsto y$, where $y$ satisfies $e(y,u)=0$. In the following we prefer to state our hypotheses directly in terms of the function $e$ in order to facilitate the use of the presented results in practice. A detailed discussion of a particular choice of $e$, for which no constraints are required to obtain a well-posed forward problem, is provided in \Cref{sec:estimation_of_convection_term}. The lower level problems are given by
\begin{equation}\label{eq:nonlinear_lower_problem}\tag{\mbox{$P(\sigma)$}}
\min_{(y,u) \in Y \times \Hs}  \| y - \ydelta \|_\Ltwo^2 + |u|_{\sigma,s}^2  \quad \text{s.t.\ } \quad u \in \Uad \text{ and } e(y,u) = 0.
\end{equation}
Here, $\ydelta \in \Ltwo$ is a noisy measurement of the ground truth state. We let
\begin{equation*} 
\Fad \coloneqq \left\{(y,u) \in Y \times \Hs \mid u \in \Uad \text{ and } e(y,u) = 0 \right\}
\end{equation*}
denote the feasible set of the lower level problem. Moreover,
\begin{equation*}
\mF \coloneqq \left\{ (\sigma, y,u) \in L^\infty((0,d)) \times Y \times \Hs \mid \sigma \in \Wad \text{ and } (y,u) \text{ solves } \eqref{eq:nonlinear_lower_problem} \right\}
\end{equation*}
denotes the solution set of the lower level problems. The learning problem is
\begin{equation}\tag{BP}\label{prob:learning_problem_nonlinear}
\min_{(\sigma,y,u) \in \Wad \times Y \times \Uad}  \frac{1}{2} \| u - \uex \|_\Ltwo^2 + \wR(\sigma) \quad \text{subject to} \quad (\sigma,y,u) \in \mF.
\end{equation}
As in the linear case, $\uex \in \Ltwo$ is the ground truth parameter and $R \colon L^\infty((0,\diam)) \to [0,\infty)$ represents an additional regularization operator for the weight.

\subsection{Existence of solutions}

\begin{definition}[Stability]
We say that $\{$\ref{eq:nonlinear_lower_problem}$\}_{\sigma \in \Wad}$ is stable (resp. weak$^*$-to-weak stable) if the following holds: \ref{eq:nonlinear_lower_problem} has a solution for every $\sigma \in \Wad$, and for every sequence $(\sigma^n)$ in $\Wad$ such that $\sigma^n \to \sigma$ in $L^\infty((0,d))$ (resp. $\sigma^n \rightharpoonup^* \sigma$ in $L^\infty((0,d))$) for some $\sigma \in \Wad$, it follows that every sequence of corresponding solutions to \ref{eq:nonlinear_lower_problem}  has a strong (resp. weak) accumulation point, and every such accumulation point is a solution to \ref{eq:nonlinear_lower_problem}.

\end{definition}

\begin{theorem}[Existence of solutions]
Assume that \ref{eq:nonlinear_lower_problem} is weak$^*$-to-weak stable and $R$ is weak$^*$ sequentially lower semicontinuous. Then 
\eqref{prob:learning_problem_nonlinear} has a solution. 
\end{theorem}
\begin{proof}
Since $\Wad$ is clearly weak$^*$ sequentially compact, weak$^*$-to-weak stability of \ref{eq:nonlinear_lower_problem} implies that $\m{F}$ is weak$^*$-weak sequentially compact. The claim now follows from the well-known fact that a function which is sequentially lower semicontinuous with respect to some topology, attains a minimum on a nonempty set which is sequentially compact with respect to the same topology. 
\end{proof}

\subsection{Optimality conditions}
We derive necessary optimality conditions for the learning problem \eqref{prob:learning_problem_nonlinear}. Here we essentially follow the discussion provided in \cite[Section 5]{Holler2018}. Throughout this section it is assumed that there exists an open neighborhood $V$ of $\Uad$ such that the function $e$ describing the state constraint is well defined and at least once continuously F-differentiable on $Y \times V$. We begin by recalling the Karush-Kuhn-Tucker conditions of the lower level problem. 
\begin{definition}[Karush-Kuhn-Tucker conditions]\label{definition:KKTconditions}
We say that $(y,u, \lambda) \in Y \times \Uad \times Z\dual$ satisfies the Karush-Kuhn-Tucker (KKT) conditions of \eqref{eq:nonlinear_lower_problem} if  
\begin{subequations}
\begin{align}
- \left[ \m{L}^s(\sigma) u  + \lambda D_u e(y,u) \right] &\in \Uad(u)^0, \label{eq:abstract_nonlinear_optimality_llp}\\
\langle y-\ydelta, w \rangle_\Ltwo + \langle \lambda D_y e(y,u), w \rangle &= 0, \quad \text{for all } w \in Y, \label{eq:abstract_nonlinear_adjoint_llp}\\
e(y,u) &= 0, \quad \text{in } Z. \label{eq:abstract_nonlinear_state_llp}
\end{align}
\end{subequations}
Here $\Uad(u) \coloneqq \{ \lambda(v- u) \mid \lambda \geq 0,  v \in \Uad \}$ is the conical hull of $\Uad - \{u\}$ and $\Uad(u)^0 \coloneqq \{ \lambda \in \Hs\dual \mid \langle f, w \rangle \leq 0  \text{ for all } w \in \Uad(u) \} $ is the polar cone of $\Uad(u)$.
\end{definition}
The KKT conditions constitute a system of first order necessary optimality conditions provided a suitable regularity assumption is met. More precisely, if $(y,u) \in Y \times \Uad$ solves \eqref{eq:nonlinear_lower_problem} and $D_y e(y,u) \in \m{L}(Y,Z)$ is bijective, then there exists a unique $\lambda \in Z \dual$  such that $(y, u, \lambda)$ satisfies the KKT conditions of 
\eqref{eq:nonlinear_lower_problem}. Here existence of $\lambda $ follows from \cite[Theorem 3.1]{Zowe1979} and proving uniqueness is straightforward. Next, we recall the Lagrange function $L \colon \Wad \times Y \times \Uad \times Z\dual \to \R$ of the lower level problem given by
\begin{equation*}
(\sigma, y, u, \lambda) \mapsto \| y - \ydelta \|_\Ltwo^2 + |u|_{\sigma,s}^2 + \lambda e(y, u) 
\end{equation*}
for $(\sigma, y, u, \lambda) \in \Wad \times Y \times \Uad \times Z\dual$,
which enables us to write second order sufficient optimality conditions in a compact form. Second order sufficient optimality conditions have many important practical implications, see \eg \cite{CasasTroeltzsch2015_Second} and the references given therein. In particular, they are closely related to stability properties of the solution mapping, see
 \cite[Chapter 2]{ItoKunisch2008_Lagrange}. It is thus of no surprise that second order sufficient conditions of the lower level problem are important for the derivation of optimality conditions for the learning problem. From now on, we assume that $e$ is at least twice continuously F-differentiable in an open neighbourhood of $Y \times \Uad$.

\begin{definition}[second order sufficient optimality condition]
We say that a point $(y,u) \in Y \times \Uad$ satisfies the second order sufficient optimality condition of \eqref{eq:nonlinear_lower_problem} if $D_y e(y,u) \in \m{L}(Y,Z)$ is bijective and there exist $\lambda \in Z\dual$ 
and $\mu > 0$ such that $(y,u,\lambda)$ satisfies the KKT conditions, and
\begin{equation*}
D^2_{(y,u)}L(\sigma, y,u, \lambda) [(\delta_y, \delta_u),(\delta_y, \delta_u)] \geq \mu \|(\delta_y, \delta_u) \|_{\He \times \Hs}^2 
\end{equation*} 
for all $(\delta_y, \delta_u) \in \ker D e(y,u) \cap (Y \times (\Uad  -\Uad))$.
\end{definition}
The constraint that feasible points must be solutions to a lower level problem prevents the direct use of Lagrangian based approaches for obtaining optimality conditions. To overcome this issue, at least to some extend, one usually considers the KKT reformulation of  bilevel optimization problem, in which the lower level problem is replaced by its KKT conditions, see \eg \cite[Section 5.5]{Dempe2002}. The KKT reformulation of the learning problem is given by
\begin{equation}\label{prob:KKT_reformulation}\tag{BP*}
\begin{split}
\min  \frac{1}{2} \| &u - \uex \|_{\Ltwo}^2 + R(\sigma) \quad \\ \text{subject to } (&\sigma, y, u, \lambda) \in \Wad \times Y \times \Uad \times Z\dual \text{ satisfies \eqref{eq:abstract_nonlinear_optimality_llp}-\eqref{eq:abstract_nonlinear_state_llp}.}
\end{split}
\end{equation}
In general, the constraints of \eqref{prob:KKT_reformulation} are easier to handle than the constraints of the original problem \eqref{prob:learning_problem_nonlinear}. For example, if there are no control constraints in the lower level problem, then the constraints of the KKT reformulated problem consist only of equality and convex constraints. In order to use the KKT reformulation \eqref{prob:KKT_reformulation} to obtain optimality conditions for the learning problem \eqref{prob:learning_problem_nonlinear}, the relation between both problems needs to be investigated. Clearly, if the lower level problem is convex for every weight $\sigma \in \Wad$, then both problems are equivalent. In general, this is not the case since points satisfying the KKT conditions of the lower level problem need not be solutions to the lower level problem. Note, however, that we are only interested in \eqref{prob:KKT_reformulation} to obtain optimality conditions for the learning problem. Consequently, for our purposes it is sufficient to know under which conditions a solution to the learning problem is guaranteed to be a local solution to \eqref{prob:KKT_reformulation}. 
\begin{theorem}\label{thm:localsolution}
Let $(\sigma^*, y^*, u^*)$ be a solution to \eqref{prob:learning_problem_nonlinear} and assume that the following statements hold:
\begin{enumerate}[label=(A\arabic*)]
\item \{\eqref{eq:nonlinear_lower_problem}\}$_{\sigma \in \Wad}$ is stable with respect to the weights,
\item $(y^*,u^*)$ satisfies the second order sufficient optimality condition of \eqref{eq:nonlinear_lower_problem} for $\sigma=\sigma^*$,
\item $(y^*,u^*)$ is the unique solution to \eqref{eq:nonlinear_lower_problem} for $\sigma = \sigma^*$.
\end{enumerate}
Then there is a unique $\lambda^* \in Z\dual$ such that $(\sigma^*, y^*, u^*, \lambda^*)$ is a local solution to \eqref{prob:KKT_reformulation}.
\end{theorem}
\begin{proof}
The proof is the same as for \cite[Theorem 5.1]{Holler2018} if one replaces $\alpha^*$ and the interval $[\ubar{\alpha}, \bar{\alpha}]$ by $\sigma^*$ and $\Wad$, respectively. 
\end{proof}

As a direct consequence of the above theorem, we get the following: Any solution to \eqref{prob:learning_problem_nonlinear}, for which the assumptions of \Cref{thm:localsolution} hold, satisfies the optimality conditions to be a local solution to \eqref{prob:KKT_reformulation}. Unfortunately, the derivation of optimality conditions for \eqref{prob:KKT_reformulation} which are convenient for numerical realization, are still impeded by the presence of control constraints in the lower level problem, which in turn lead to set valued constraints in \eqref{prob:KKT_reformulation}. Issues involving such constraints seem to be not yet fully resolved (at least in the infinite dimensional case) and are subject to ongoing research, see \eg \cite{Harder2018}. For this reason, we only consider the case without constraints in the following lemma. 

\begin{lemma} Assume that $\Uad = \Hs$. Let $(\sigma^*, y^*,u^*,\lambda^*)$ be a local solution to \eqref{prob:KKT_reformulation} with $(y^*,u^*)$ satisfying the second order sufficient optimality condition of \eqref{eq:nonlinear_lower_problem} for $\sigma = \sigma^*$. Then there is a unique $(p^*,q^*,z^*) \in Y \times \Hs \times Z\dual$ such that
\begin{subequations}
\begin{align}
D_\sigma R(\sigma^*)[\sigma - \sigma^*] + [\m{L}^s(\sigma-\sigma^*)u^*] q^* \geq 0, \quad \forall \sigma \in \Wad,\label{eq:nonlinear_optimality}  \\
p^* + \lambda^* e_{yy}(y^*,u^*) p^* +\lambda^* e_{yu}(y^*,u^*) q^*+z^* e_y(y^*,u^*)=0, \label{eq:nonlinear_adjoint1} \\
u^*-\uex + \lambda^* e_{uy}(y^*,u^*) p^* + \m{L}^s(\sigma^*) q^* + \lambda^* e_{uu} (y^*,u^*) q^* + z^* e_u(y^*,u^*)=0, \label{eq:nonlinear_adjoint2}\\
e_y(y^*,u^*) p^*   + e_u(y^*,u^*) q^* =0. \label{eq:nonlinear_adjoint3}
\end{align}
\end{subequations}
Here we write $e_{yy}(y^*,u^*) \coloneqq D_y^2 e(y^*,u^*)$, $e_{yu} \coloneqq D_y D_u e (y^*,u^*)$ and analogously for the other partial derivatives.  
\end{lemma}
\begin{proof}
The proof is the same as for \cite[Lemma 5.1]{Holler2018} if one replaces $\alpha^*$ and the interval $[\ubar{\alpha}, \bar{\alpha}]$ by $\sigma^*$ and $\Wad$, respectively. 
\end{proof}

\subsection{Estimation of the convection term}\label{sec:estimation_of_convection_term}
We consider the problem of estimating a vector valued convection term $b^\dagger \in \Hs^N$ in an elliptic PDE based on a noisy observation $\ydelta \in \Ltwo$ of the ground truth state $\yex \in \Hen$. The function $e \colon \Hen \times \Lq^N \to \Hend$ describing the PDE is in its weak form given by
\begin{equation}\label{eq:advection}
e(y,b) w = \int_\Omega \grad y \cdot \grad w + b \cdot (w\grad y)  +  c y w - f w \dx, 
\end{equation}
for $(y,b,w) \in \Hen \times \Lq^N \times \Hen$. Here $f \in \Ltwo$ is a nonzero given source term, $c \in \Linfty$ is a given potential term, which is assumed to be nonnegative almost everywhere, and $1 < q < \infty$. We restrict ourselves to dimension $N \in \{1,2,3\}$. As we will see (\Cref{prop:Lq_stability_state_equation_examplemulti}), we need to require $q>2$ if $N \in \{1,2\}$ and $q \geq 3$ if $N=3$ to ensure that $e$ is well-defined, and that the PDE $e(y,b)=0$ has a unique solution $y \in \Hen$ for every $b \in \Lq^N$. Since we are interested in the case $b \in \Hs^N$, we often require that $\Hs$ is compactly embedded in $\Lq$ for $q$ satisfying the above requirements. To achieve this, we frequently make the following assumption:
\begin{enumerate}[label=(B)]
\item  If $N \in \{ 1, 2\}$ then $s \in (0,1)$ and if $N=3$ then $s \in (1/2,1)$.\label{assumption:convection_relation_dim_s}
\end{enumerate}
We emphasize that making use of results from \cite{CasasMateosRoesch2019} we neither assume that $\text{div} \, b=0$ nor that $b$ is small in the $L^\infty(\Omega)^N$ norm. The lower level problem is given by
\begin{equation}\tag{\mbox{$\text{P}_\text{adv}(\sigma)$}}\label{prob:problem_convection}
\min_{(y,b) \in \Hen \times \Hs^N} \| y - \ydelta \|_\Ltwo^2 + | b |_{\sigma,s}^2 \quad \text{subject to } \quad e(y,b) = 0.
\end{equation}
The definition of the nonlocal energy seminorm $|\cdot|_{\sigma,s}$ is thereby extended to vector valued functions $b = (b_1,\dots, b_N)^\top \in \Hs^N$ by letting
$
| b |_{\sigma,s}^2 \coloneqq \sum_{i=1}^N |b_i|_{\sigma,s}^2.
$. 
The learning problem is
\begin{equation}\tag{\mbox{$\text{BP}_\text{adv}$}}\label{prob:learning_problem_advection}
\min_{(\sigma,y,b) \in \Wad \times \Hen \times \Hs^N} \frac{1}{2} \| b -b^\dagger \|_\Ltwo^2 + \wR(\sigma) \quad \text{subject to} \quad (\sigma, y,b) \in \mF.
\end{equation}
We begin by verifying that $e$ is well-defined and stable with respect to the state and convection term. 
\begin{proposition}\label{prop:Lq_stability_state_equation_examplemulti}
Let $e$ be as in \eqref{eq:advection} with $q > 2$ if $N \in \{ 1, 2\}$ and $q \geq 3$ if $N=3$. Then
\begin{enumerate}[label=\roman*)]
\item $e$ is well-defined, infinitely many times F-differentiable, and (weak, strong)-to-weak sequentially continuous as a mapping from $\Hen \times \Lq^N$ to $\Hend$,\label{item_conv_weak_strong_continuity}
\item for every $b \in \Lq^N$ there is a unique $y(b) \in \Hen$ such that $e(y(b),b)=0$,\label{item:uniqueness_sol_conv}
\item  $D_y e(y,b)$ is bijective for every $(y,b) \in \Hen \times \Lq^N$,\label{item:uniqueness_sol_linearized_conv}
\item the mapping $b \mapsto y(b)$ such that $e(y(b),b)=0$ is continuously F-differentiable as a mapping from $\Lq^N$ to $\Hen$. \label{item:cont_F_differentiability}

\end{enumerate}

\end{proposition}
\begin{proof}
\hfill
\begin{enumerate}
\item [\ref{item_conv_weak_strong_continuity}] Since the affine part of $e$, which depends only on the state, is clearly well-defined, infinitely many times F-differentiable, and (weak, strong)-to-weak sequentially continuous, it remains to verify \ref{item_conv_weak_strong_continuity}
 with $e$ replaced by the bilinear part $B \colon \Hen \times \Lq^N \to \Hend$ of $e$, which is given by
\begin{equation*}
B(y,b)w \coloneqq \int_\Omega \sum_{i=1}^N  b_i(x) w(x) \partial_i y(x) \dx 
\end{equation*}
for every $(y,b,w) \in \Hen \times \Lq^N \times \Hen$. Using classical Sobolev embeddings (see \eg \cite[Theorem 4.12]{adams2003sobolev}) and the assumption on $q$,  there exists $1 < p < \infty$ with $1/q + 1/p = 1/2$ such that $\Hen$ is continuously embedded in $\Lp$. Applying H\"older's inequality, we estimate
\begin{multline*}
|B(y,b)w| \leq  \sum_{i=1}^N \| b_i \|_\Lq  \| \partial_i y \|_{\Ltwo} \| w \|_{\Lp} \\ \leq  C \| b \|_{\Lq}  \| y \|_{\He} \| w \|_{\He},
\end{multline*}
for a suitable constant $C >0$. This proves that $B$ is well-defined and continuous. Since bilinear continuous functions are always infinitely many times F-differentiable and (weak, strong)-to-weak continuous, this concludes the proof of \ref{item_conv_weak_strong_continuity}.
\item [\ref{item:uniqueness_sol_conv}--\ref{item:uniqueness_sol_linearized_conv}]This follows from \cite[Theorem 2.1]{CasasMateosRoesch2019}.
\item [\ref{item:cont_F_differentiability}] Using the first three assertions, the claim follows from the implicit function theorem (see \eg \cite[Theorem 4.7.1]{Cartan1971}).
\end{enumerate}

\end{proof}
\begin{corollary}\label{cor:stability_of_conv_state}
Let $e$ be as in \eqref{eq:advection} and let $(N,s)$ satisfy \ref{assumption:convection_relation_dim_s}. Then
\begin{enumerate}[label=\roman*)]
\item $e$ is infinitely many times continuously F-differentiable and (weak, weak)-to-weak sequentially continuous as a mapping from $\Hen \times \Hs^N$ to $\Hend$,\label{item_conv_weak_continuity}
\item for every $b \in \Hs^N$ there is a unique $y(b) \in \Hen$ such that $e(y(b),b)=0$,\label{item:conv_uniqueness_sol}
\item  $D_y e(y,b)$ is bijective for every $(y,b) \in \Hen \times \Hs^N$,
\item the mapping $b \mapsto y(b)$ such that $e(y(b),b)=0$ is continuously F-differentiable as a mapping from $\Hs^N$ to $\Hend$,\label{item:conv_cont_diffb}
\item the mapping $b \mapsto y(b)$ such that $e(y(b),b)=0$ is Lipschitz continuous as a mapping from $\Hs^N$ to $\Hen$ on every bounded subset of $\Hs^N$. \label{item:conv_sol_lipschitz}

\end{enumerate}

\end{corollary}
\begin{proof}
If $N \in \{1,2\}$, then for all $s \in (0,1)$ there is $q > 2$ such that $\Hs$ is compactly embedded in $\Lq$. If $N=3$ and $s \in (1/2,1)$, then there is $q \geq 3$ such that $\Hs$ is compactly embedded in $\Lq$. Combining this observation with \Cref{prop:Lq_stability_state_equation_examplemulti}, the first four assertions follow easily. To prove  \ref{item:conv_sol_lipschitz}, it suffices to prove that the mapping is Lipschitz continuous on every ball $B_r(\hat{b}) \coloneqq \{ b \in \Hs^N  \mid \| b - \hat{b} \|_{\Hs} \leq r$  \} with radius $r> 0$ and center $\hat{b} \in \Hs^N$. Observe that due to \ref{item:conv_cont_diffb} and the compactness of the embedding of $\Hs^N$ into $\Lq^N$ we have  $M \coloneqq \sup_{b \in B_x(r)} \| D y(b) \|_{\m{L}(\Lq^N,\He)}$ is finite. Moreover, it is easy to see that
 \begin{equation*}
 \| D y(b) \|_{\m{L}(\Hs,\He)} \leq C_{s,q} \| D y(b) \|_{\m{L}(\Lq,\He)} \leq C_{s,q} M \quad \text{for all } b \in B_x(r)
\end{equation*}
where $C_{s,q}$ denotes the embedding constant of $\Hs$ into $\Lq$. It now follows from \cite[Theorem 3.3.2]{Cartan1971} that $b \mapsto y(b)$ is Lipschitz continuous on $B_x(r)$ (with Lipschitz constant bounded by $C_{s,q} M$).
\end{proof}
The following technical result is needed for the existence proof in \Cref{prop:existence_of_solutions_llp_advection}.
\begin{lemma}\label{prop:partial_derivative_of_Hen_zero_implies_zero}
Let $y \in \Hen$ and $v \in \R^N \setminus \{0\}$ be such that $v \cdot \grad y = 0$ \ale on $\Omega$. Then $y = 0$.
\end{lemma}
\begin{proof}
Let $\tilde{y}$ denote the zero extension of $y$ to the complement of $\Omega$ in $\R^N$. It is well-known, see \eg \cite[Lemma 3.27 on p. 71]{adams2003sobolev},  that $\tilde{y} \in H^1(\R^N)$ and
\begin{equation*}
\grad \tilde{y}(x) = 0 \quad \text{\ale on } \R^N \setminus \Omega.
\end{equation*}
 Consequently, we have $v \, \cdot \, \grad \tilde{y}(x) = 0$ \ale on $\R^N$. Let $(\rho^n)$ in $C^\infty(\R^N)$ be a sequence of mollifiers as defined in \cite[p. 109] {Brezis2010}. Define the sequence $(w^n)$ by $w^n \coloneqq \rho^n \conv \tilde{y}$ for every $n \in \N$. It follows from \cite[Proposition 4.20 on p. 107 and Lemma 9.1 on p. 266]{Brezis2010} that $w^n \in C^\infty(\R^N)$ and
\begin{equation*}
v \cdot \grad w^n =  \rho^n \conv (v \cdot \grad \tilde{y}) = 0.
\end{equation*} 
Moreover, $w^n$ has compact support, since $\rho^n$ and $\tilde{y}$ have compact support (see \cite[Proposition 4.18 on p.106]{Brezis2010}). Consequently, for arbitrary $x \in \R^N$ there exists $\alpha > 0$ such that $x + \alpha v \notin \text{supp} (w^n)$. We have
\begin{equation*}
0 = \int_0^\alpha v \cdot \grad w^n ( x + s v ) \, \mathrm{d} s = w^n(x+\alpha v ) - w^n(x) =  - w^n(x).
\end{equation*}
Since $x \in \R^N$ was arbitrary, this proves that $w^n$ is zero on $\R^N$. Since $(w^n)$ also converges to $\tilde{y}$ in $\Ltworn$ as $n \to \infty$ (see \cite[Theorem 4.22 on p. 109]{Brezis2010}) this implies that $y$ is zero. 
\end{proof}

\begin{proposition}\label{prop:existence_of_solutions_llp_advection}
Let $(N,s)$ satisfy \ref{assumption:convection_relation_dim_s}. Then \eqref{prob:problem_convection} has a solution if and only if there exists $(y,b) \in \Fad$ such that
\begin{equation}\label{eq:cond_existene_solutions_inv_convection}
\| y - \ydelta \|_\Ltwo^2 + |b|_{\sigma,s}^2  \leq \| \ydelta\|_\Ltwo^2.
\end{equation}
\end{proposition}

\begin{proof}
To prove that the existence of $(y,b) \in \Fad$ satisfying \eqref{eq:cond_existene_solutions_inv_convection} is necessary for existence of solutions, it suffices to show that
\begin{equation}\label{eq:exist_proof_conv_inf_est}
\inf_{(y,b) \in \Fad } \| y - \ydelta \|_\Ltwo^2 + |b|_{\sigma,s}^2  \leq \| \ydelta\|_\Ltwo^2.
\end{equation}
To do this, let $(b^n)$ be a sequence of constant functions in $\Hs^N$ such that $\|b^n\|_{\Hs}$ diverges to $\infty$ as $n \to \infty$. Interpreting $(b^n)$ as a sequence of vectors in $\R^N$, we have $|b^n| \to \infty$. Using that $\text{div} \, {b^n} = 0$ for every $n \in \N$, it is straightforward to prove that $(y^n)$ is bounded in $\He$. Let $v^n = b^n/|b^n|$. Then there exists subsequences of $(y^n)$ and $(v^n)$, again denoted by $(y^n)$ and $(v^n)$, such that
\begin{equation*}
y^n \rightharpoonup y \text{ in } \Hen \quad \text{and} \quad v^n \to v \text{ in } \R^N.
\end{equation*}
Let $\phi \in \Testfunctions$ be arbitrary. Testing $e(y^n, b^n)$ with $w^n = \phi / |b^n|$ yields
\begin{equation*}
\int_\Omega v \cdot (\phi \grad y) \mathrm{d}x = \lim_{n \to \infty} \int_\Omega v^n \cdot (\phi \grad y^n) \mathrm{d}x  =  \lim_{n \to \infty}   -\frac{1}{|b^n|} \int_\Omega \grad y^n \cdot \grad \phi + c y^n\phi - f \phi \mathrm{d}x= 0.
\end{equation*} 
Consequently, by the fundamental lemma of the calculus of variations it follows that $v \cdot \grad y = 0$ \ale on $\Omega$. By \Cref{prop:partial_derivative_of_Hen_zero_implies_zero}, we deduce that $y = 0$. It follows that 
\begin{equation}
\inf_{(y,b) \in \Fad } \| y - \ydelta \|_\Ltwo^2 + |b|_{\sigma,s}^2  \leq \lim_{n \to \infty} \|y^n - \ydelta\|_\Ltwo^2  = \| \ydelta \|_\Ltwo^2,
\end{equation}
where we use that for every $n \in \N$ we have $|b^n|_{\sigma,s} = 0$ (since $b^n$ is constant). This proves \eqref{eq:exist_proof_conv_inf_est}, which in turn implies that the existence of $(y,b) \in \Fad$ satisfying \eqref{eq:cond_existene_solutions_inv_convection} is necessary for the existence of solutions to \eqref{prob:problem_convection}. 

We now prove that the existence of $(y,b) \in \Fad$ satisfying \eqref{eq:cond_existene_solutions_inv_convection} is also sufficient to guarantee existence of solutions. It follows from \Cref{cor:stability_of_conv_state} \ref{item:conv_uniqueness_sol} that the feasible set $\Fad$ is nonempty. Consequently, we can take a minimizing sequence $(y^n, b^n)$ in $\Fad$ to \eqref{prob:problem_convection}. We divide the proof into three steps. 
\begin{enumerate}
\item In the first step we prove that $(y^n)$ is bounded in $\He$. Since $(y^n, b^n)$ is a minimizing sequence to \eqref{prob:problem_convection}, using the cost functional in \eqref{prob:problem_convection}, it can be easily derived that
\begin{equation}\label{eq:bounded_bn2}
(|b^n|_{\sigma,s}) \text{ and } (\| y^n \|_\Ltwo) \text{ are bounded.}
\end{equation}
For every $n \in \N$  we can write $b^n$ as the sum of a constant function and a function with mean value zero, \ie $b^n = b_1^n + b_2^n$ in $\Pi^0(\Omega)) + \Pi^0(\Omega)^\perp$. It then follows from \eqref{eq:bounded_bn2}, the fact that $|b^n|_{\sigma,s}=|b_2^n|_{\sigma,s}$ for all $n \in \N$, \Cref{lem:wirtinger}, and \Cref{cor:equivalenceofnorms} that $(b_2^n)$ is bounded in $\Hs^N$. Using the chain rule and integration by parts we moreover have
\begin{equation}\label{eq:vanishing_constant_part}
\int_\Omega b_1^n \cdot (y^n \grad y^n ) \, \dx = \frac{1}{2} \int_\Omega b_1^n \cdot ( \grad (y^n)^2 ) \, \dx =  \frac{1}{2} \int_{\partial \Omega}  y^2 (b_1^n \cdot \boldsymbol{\nu}) \,  \mathrm{d}^{N-1} = 0,
\end{equation}
since $b_1^n$ is constant and $y \in \Hen$. Using \eqref{eq:vanishing_constant_part}, testing $e(y^n,b^n) = 0$ with $y^n$ yields
\begin{equation*}
\int_\Omega |\grad y^n |_2^2 \, \dx = - \int_\Omega \left( b_2^n \cdot (y^n \grad y^n) + y^2 - f y^n \right) \dx.
\end{equation*}
Applying Poincar\'e's and H\"older's inequality, and using that $(y^n)$ is bounded in $\Ltwo$, we deduce that
\begin{multline}\label{eq:estimate1_ex_sol_convection}
c \| y^n \|_\He^2 \leq \int_\Omega |\grad y^n |_2^2 \dx  = - \int_\Omega \left( b_2^n \cdot (y^n \grad y^n) + c y^2 - f y^n \right) \dx \\ \leq \| b_2^n \|_{\Lq} \| y^n \|_\Lp \| y^n \|_\He  + C,
\end{multline}
where $c,C>0$ are suitably chosen constants and $1 < p, q < \infty$ are such that $1/p+1/q=1/2$, $\Hs$ is compactly embedded in $\Lq$, and $\Hen$ is compactly embedded in $\Lp$. If $(y^n)$ is bounded in $\Lp$, then it follows from \eqref{eq:estimate1_ex_sol_convection}  that $\| y^n \|_\He$ is bounded and the first step is finished. If $(y^n)$ is not bounded in $\Lp$, it has a subsequence, again denoted by $(y^n)$, such that $\| y^n \|_\Lp$ tends to $\infty$ as $n \to \infty$. If we define $z^n \coloneqq \frac{y^n}{\|y^n\|_\Lp}$,
then $\| z^n  \|_\Lp = 1$ and
\begin{equation*}
c \| z^n \|_\He^2 \leq \| b_2^n \|_\Lq \| z^n \|_\He + D^n,  
\end{equation*}
where $D^n \to 0$  as $n \to \infty$. Dividing this inequality by $\| z^n \|_\He$, it follows that $(z^n)$ is bounded in $\He$. Consequently, there exists a subsequence again denoted by $(z^n)$ such that $z^n \rightharpoonup z$ in $\Hen$. Since $\Hen$ is compactly embedded in $\Lp$ we have $z^n \to z$ in $\Lp$. It follows that $\| z \|_\Lp = 1$. However since $(y^n)$ is bounded in $\Ltwo$ we also have
\begin{equation*}
\| z^n \|_\Ltwo = \| y^n \|_\Ltwo (\| y^n \|_\Lp)^{-1} \to 0 \quad \text{as } n \to \infty,
\end{equation*}
which implies that $z = 0$. This is a contradiction, and consequently $(y^n)$ must be bounded in $\He$. Hence, $(y^n)$ has a subsequence, again denoted by $(y^n)$ such that $y^n \rightharpoonup y $ in $\Hen$ for some $y \in \Hen$.
\item We now prove that $(b_1^n)$ has a bounded subsequence. Here we interpret the sequence of constant functions $(b_1^n)$ as a sequence of vectors in $\R^N$. We argue by contradiction and assume that there is a subsequence, again denoted $(b_1^n)$, such that $b_1^n \neq 0$ for all $n \in \N$ and $|b_1^n| \to \infty$ as $n \to \infty$. If we let $v^n \coloneqq b_1^n /|b_1^n|$, then $(v^n)$ has a subsequence converging to some $v \in \R^N$. Let $\phi \in \Testfunctions$ be arbitrary. Testing $e(y^n,b^n)$ with $w^n \coloneqq \frac{\phi}{|b_1^n|}$ we deduce 
\begin{multline*}
\int_\Omega v \cdot (\phi \grad y) = \lim_{n \to \infty} \int_\Omega v^n \cdot (\phi \grad y^n)\\  =  \lim_{n \to \infty}  - \frac{1}{|b_1^n|} \int_\Omega \grad y^n \cdot \grad \phi + b_2^n \cdot ( \phi \grad y^n)  +  c y^n w - f w \dx = 0.
\end{multline*} 
Consequently, by the fundamental lemma of the calculus of variations it follows that $v \cdot \grad y = 0$ \ale in $\Omega$.
By \Cref{prop:partial_derivative_of_Hen_zero_implies_zero}, we deduce that $y = 0$. Consequently, we must have
\begin{equation}\label{eq:inf_larger_conv_ex}
\inf_{(y,b) \in \Fad} \| y - \ydelta \|_2^2 + |b|_{\sigma,s}^2 =              \lim_{n \to \infty} \| y^n - \ydelta \|_\Ltwo^2 + |b^n|_{\sigma,s}^2  \geq \| \ydelta \|_\Ltwo^2.   
\end{equation}
Now if \eqref{eq:inf_larger_conv_ex} really holds, then existence of solutions follows from \eqref{eq:cond_existene_solutions_inv_convection} and the proof is finished. If \eqref{eq:inf_larger_conv_ex} does not hold, 
then our assumption must have been wrong and consequently $(b_1^n)$ must be bounded.

\item In the first two steps we have established that either existence of solutions holds trivially or $(y^n,b^n)$ is bounded in $\Hen \times \Hs^N$. In the second case, the sequence $(y^n, b^n)$ has a weak accumulation point $(y,b)$ in $\Hen \times \Hs^N$. The (weak,weak)-to-weak sequential continuity established in \Cref{cor:stability_of_conv_state} \ref{item_conv_weak_continuity} implies that $e(y,b) = 0$. Since the cost functional in  \eqref{prob:problem_convection} is sequentially weakly lower semicontinuous, it follows that $(y,b)$ is a solution to \eqref{prob:learning_problem_advection}. This finishes the proof. 
\end{enumerate}
\end{proof}

\begin{remark}
A sufficient condition for  \eqref{eq:cond_existene_solutions_inv_convection} to be satisfied is that there exists a state corresponding to a constant convection coefficient, which lies in a ball with radius $\|\ydelta\|_\Ltwo$ and center $\ydelta$. 
\end{remark}

\begin{proposition}[Existence of solutions to the learning problem]
Assume that $R$ is weak$^*$ sequentially lower semicontinuous on $L^\infty((0,d))$. Then \eqref{prob:learning_problem_advection} has a solution if and only if there exist $\sigma \in \Wad$
 and $(y,b) \in \Fad$ such that 
\begin{equation}\label{eq:nec_suff_condition_bilevel}
\| y - \ydelta \|_\Ltwo^2 + | b |_{\sigma,s}^2 \leq \| \ydelta \|_\Ltwo^2.
\end{equation}
\end{proposition}

\begin{proof}
As a straightforward consequence of \Cref{prop:existence_of_solutions_llp_advection} we obtain that the feasible set $\m{F}$ of \eqref{prob:learning_problem_advection} is nonempty if and only if there exist $\sigma \in \Wad$
 and $(y,b) \in \Fad$ satisfying \eqref{eq:nec_suff_condition_bilevel}. This proves that the condition above is necessary for the existence of solutions. To prove that it is also sufficient, note that if the above condition holds, then the feasible set $\m{F}$ is nonempty. Consequently, we can take a minimizing sequence $(\sigma^n,y^n, b^n)$ in $\m{F}$  to \eqref{prob:learning_problem_advection}. It follows from the minimizing sequence property that $(b^n)$ is bounded in $\Ltwo$ and from the feasibility that $(|b^n|_{\sigma^n, s}^2)$ is bounded. In combination with \Cref{cor:equivalenceofnorms}, this yields that $(b^n)$ is bounded in $\Hs^N$. By \Cref{cor:stability_of_conv_state} \ref{item:conv_sol_lipschitz}, the boundedness of $(b^n)$ in $\Hs^N$ in turn implies that $(y^n)$ is bounded in $\Hen$. Since $\Wad$ is weak$^*$ sequentially compact, it follows that $(\sigma^n, y^n, b^n)$, has a subsequence, again denoted by  $(\sigma^n, y^n, b^n)$, such that 
$(\sigma^n, y^n, b^n)$ converges to  $(\sigma, y, b) \in \Wad \times \Hen \times \Hs^N$ in the (weak$^*$, weak, weak) sense in $L^\infty((0,d)) \times \Hen \times \Hs^N$. As a consequence of \Cref{cor:stability_of_conv_state} \ref{item_conv_weak_continuity} moreover we have $(y,u) \in \Fad$. Our next aim is to prove that $(y,u)$ solves the lower level problem with $\sigma$. To do this, first note that it can be derived from the KKT conditions in \Cref{definition:KKTconditions} that for every $i \in \{1, \dots, N\}$ there exists a sequence $(\lambda_i^n)$ in $\Hen$ such that for every $n \in \N$ \begin{subequations}
\begin{align*}
[\m{L}^s(\sigma^n) b^n_i] v  + \int_\Omega v(x) \lambda^n(x) \partial_i y^n(x) \dx  &= 0 \quad \text{for all } v \in \Hs \\
\int_\Omega (y^n-\ydelta)w + \grad \lambda^n \cdot \grad w + b^n \cdot (\lambda^n \grad w) \,  + c \lambda^n w \, \mathrm{d} x &= 0 \quad \text{for all } w \in \Hen
\end{align*}
\end{subequations}
Arguing similarly as in \Cref{cor:stability_of_conv_state} \ref{item:conv_sol_lipschitz} one can prove that $(\lambda^n_i)$ is bounded in $\Hen$. Now let $p_i^n(x) \coloneqq \lambda_i^n(x) \partial_i y^n(x)$. We must now distinguish between two cases. 
\begin{enumerate}
\item If $N \in \{1,2\}$, then $(p_i^n)$ is bounded in $L^{2-\delta}(\Omega)$ for arbitrary $0 < \delta <1$. 
\item If $N=3$ and $s \in (1/2,1)$, then $(p_i^n)$ is bounded in $L^{3/2}(\Omega) $ and $\Hs$ is compactly embedded in $L^3(\Omega)$.
\end{enumerate}
In either case, however, it follows from \Cref{lem:stability} that
\begin{equation}
\lim_{n \to \infty} | b_i^n |_{\sigma^n, s} = |b_i|_{\sigma,s} \quad \text{for every } n \in \N.
\end{equation}
This implies that  
\begin{multline*}
\| y- \ydelta\|_\Ltwo^2 + |b |_{\sigma,s}^2 = \lim_{n \to \infty} \| y^n- \ydelta\|_\Ltwo^2 + |b^n |_{\sigma^n,s}^2 \\ \leq \lim_{n \to \infty} \| w^n - \ydelta \|_\Ltwo^2 + | v |_{\sigma^n,s}^2 = \| w - \ydelta \|_\Ltwo^2 + | v |_{\sigma,s}^2,
\end{multline*}
for all $(w, v) \in \Fad$. This shows that $(\sigma, y, b ) \in \m{F}$. Since by assumption the cost functional is weak$^*$-weak-weak sequentially lower semi continuous, it follows that $(\sigma, y, b )$ solves \eqref{prob:problem_convection}, which finishes the proof. 
\end{proof}

\subsubsection{Sufficient optimality conditions and uniqueness of solutions of the lower level problem}
In view of \Cref{thm:localsolution}, it is important to know under which circumstances a computed KKT point of the lower level problem is the lower level problem's unique solution and satisfies a second order optimality condition. A satisfactory answer to this question is beyond the scope of this paper. Nevertheless, in order to gain at least some insight into the problem, in the following we provide conditions which are sufficient (but not necessary) to ensure that a computed KKT point of the lower level problem is indeed its unique solution and satisfies a second order sufficient optimality condition. These conditions are essentially smallness assumptions on the adjoint state. We remark that the provided estimates are of qualitative nature and quantification would require additional arguments. Conditions similar to ours, which are sufficient for the uniqueness of solutions to a class of semilinear elliptic equations with the controls entering linearly have been discussed in \cite{Ali2016}. 

For the sake of simplicity, our discussion is limited to the case of a constant weight with value $\alpha > 0$ and $c \equiv 0$ in \eqref{eq:advection}. In this case, the lower level problem can be written as 
\begin{equation}\label{prob:conv_simplified}
\min_{(y,b) \in \Hen \times \Hs}  \| y - \ydelta \|_\Ltwo^2 + \alpha | b |_{\Hs}^2 \quad \text{subject to} \quad e(y,b) = 0.  
\end{equation}
Throughout this section $(\bar{y},\bar{b},\bar{\lambda})$ denotes a KKT point of \eqref{prob:conv_simplified}. Recall that $(\bar{y},\bar{b},\bar{\lambda})$ is said to satisfy the second order sufficient optimality conditions of \eqref{prob:conv_simplified} if there exists $\mu > 0$ such that
\begin{equation*}
D^2_{(y,b)}L(\alpha,\bar{y},\bar{b},\bar{\lambda}) [(\delta_y, \delta_b),(\delta_y, \delta_b)] \geq \mu \|(\delta_y, \delta_b) \|_{\Ltwo \times \Hs}^2 
\end{equation*} 
for all $(\delta_y, \delta_b) \in \ker D e(\bar{y},\bar{b})$. Straightforward computations show that
\begin{equation}\label{eq:Lagrange_second_simplified_convection}
D^2_{(y,b)}L(\sigma,\bar{y},\bar{b},\bar{\lambda}) [(\delta_y, \delta_b),(\delta_y, \delta_b)] = \| \delta_y \|_\Ltwo^2 + \alpha \| \delta_b \|_{\Hs}^2 + \int_\Omega \delta_b \cdot (\bar{\lambda} \grad \delta_y ) \dx.
\end{equation}
The following technical result will be needed later on.
\begin{lemma}\label{lem:boundedness_of_solution_op_lin_pde}
Let $(N,s)$ satisfy \ref{assumption:convection_relation_dim_s}. Then there is a constant $C$ (depending on $\bar{b}$) such that for all $(\delta_y,\delta_b) \in \ker D e(\bar{y},\bar{b})$ we have
\begin{equation*}
\| \delta_y \|_\He \leq C \| \grad y \|_{\Ltwo^N} \| \delta_b \|_\Hs .
\end{equation*}
\end{lemma}
\begin{proof}
Note that $(\delta_y, \delta_b) \in \ker D e(\bar{y},\bar{b})$ if and only if $(\delta_y,\delta_b)$ is a weak solution to 
\begin{equation*}
- \Delta \delta_y + \bar{b} \cdot \grad \delta_y = - \delta_b \cdot \grad \bar{y} \quad \text{in } \Hend.
\end{equation*}
Consequently, by \cite[Theorem 2.1]{CasasMateosRoesch2019} there is $C_1>0$ (depending on $\bar{b}$) such that
\begin{equation}\label{eq:deltaydeltab_continv}
\| \delta_y \|_\He \leq C_1 \| \delta_b \cdot \grad \bar{y} \|_{\Hend}.
\end{equation}
Since $N$ and $s$ satisfy \ref{assumption:convection_relation_dim_s}, there are $ p,q > 1$ such that $1/p+1/q = 1/2$, and $\Hs$ and $\He$ are continuously embedded in $\Lq$ and $\Lp$, respectively. Using H\"older's inequality, it follows that 
\begin{equation}\label{eq:deltaydeltab_HoelderSov}
\| \delta_b \cdot \grad \bar{y} \|_{\Hend} \leq C_{s,p} C_{1,q} \| \grad \bar{y} \|_{\Ltwo^N} \| \delta_b \|_\Hs
\end{equation}
where $C_{s,q}$ and $C_{1,p}$ denote the embedding constants of $\Hs$ and $\He$ into $\Lq$ and $\Lp$, respectively. Combining \eqref{eq:deltaydeltab_continv} and \eqref{eq:deltaydeltab_HoelderSov} yields the desired estimate.
\end{proof}
We now formulate a criterion on the adjoint state, which implies that $(\bar{y},\bar{b},\bar{\lambda})$ satisfies the second order sufficient optimality conditions.
\begin{proposition}
Let $(N,s)$ satisfy \ref{assumption:convection_relation_dim_s},  and let $p,q>1$ be such that $1/p+1/q = 1/2$ and $\Hs$ is continuously embedded in $\Lq$. Let $(\bar{y}, \bar{b} , \bar{\lambda})$ be a KKT point of \eqref{prob:conv_simplified}. Then there exists a constant $C > 0$ (depending on $(\bar{y}, \bar{b})$) such that if
\begin{equation*}
\| \bar{\lambda} \|_\Lp < r,
\end{equation*} then $(\bar{y}, \bar{b})$ satisfies the second order sufficient optimality conditions of \eqref{prob:conv_simplified}.
\end{proposition}
\begin{proof}
We divide the proof into several steps.
\begin{enumerate}
\item We begin by showing that there is a constant $\kappa_1 > 0$ such that for every $v \in \Pi_0(\Omega)$ and $w \in \Hen$ with $(w,v) \in \ker D e(y,b)$ we have $\|w\|_\Ltwo^2 \geq \kappa_1 \| v \|_\Ltwo^2$. Since $\Pi_0(\Omega)$ is finite dimensional and the mapping $v \mapsto w(v)$ such that $(w(v),v) \in \ker D e (y,b)$ is linear, it suffices to show that $v \neq 0$ implies that $w(v) \neq 0$. Note that $v \cdot \grad y \neq 0 $ since otherwise $y$ would also have to be zero. Since $w$ satisfies 
\begin{equation*}
-\Delta w + b \cdot \grad w = - v \cdot \grad y,
\end{equation*}
it follows that $w$ is not equal to zero. 
 \item In the second step we prove that for every $c > 0$ there exists a constant $\kappa > 0$ such that for all $(\delta_y, \delta_b) \in \ker D e(y,u)$ it holds that
\begin{equation}\label{eq:seminorm_L2statecontrol}
\| \delta_y \|_\Ltwo^2 + c | \delta_b|_\Hs^2  \geq \kappa \| \delta_b \|_\Hs^2.
\end{equation}
We can write $\delta_b = \delta_b^1 + \delta_b^2$, where $(\delta_b^1, \delta_b^2) \in \Pi_0(\Omega) \times \Pi_0(\Omega)^\perp$. Denote by $\delta_y^1$ and $\delta_y^2$ the unique elements in $\Hen$ such that $(\delta_y^1, \delta_b^1) \in \ker D e(y,u) $ and 
$(\delta_y^2, \delta_b^2) \in \ker D e(y,u) $, respectively. By linearity we have $\delta_y = \delta_y^1 + \delta_y^2$. From the first step, we know that there exists a constant $\kappa_1 > 0$ such that  
\begin{equation*}
\|\delta_y^1 \|_\Ltwo^2 \geq \kappa_1 \|\delta_b^1 \|_\Ltwo^2.
\end{equation*}
Moreover, combining \Cref{lem:boundedness_of_solution_op_lin_pde} and \Cref{lem:wirtinger} we also have
\begin{equation*}
\|\delta_y^2 \|_\Ltwo^2 \leq \kappa_2 \|\grad y\|_{\Ltwo^N} |\delta_b^2 |_{\Hs},
\end{equation*}
where $\kappa_2>0$ is independent of $(\delta_y, \delta_b)$. For $\beta \in (0,1)$ we now estimate
\begin{multline*}
\| \delta_y \|_\Ltwo^2 + c | \delta_b |_\Hs^2 \\ = \|\delta_y^1 \|_\Ltwo^2 + \|\delta_y^2 \|_\Ltwo^2 + 2 \langle \sqrt{\beta} \delta_y^1, (1/\sqrt{\beta}) \delta_y^2 \rangle_\Ltwo  + c |\delta_b^2 |_\Hs \\ \geq (1-\beta) \| \delta_y^1 \|_\Ltwo^2 + \left(1- 1/\beta \right) \|\delta_y^2 \|_\Ltwo^2 + c | \delta_b^2 |_{\Hs} \\  \geq (1-\beta) \kappa_1 \| \delta_b^1 \|_\Ltwo^2 + \left(\left(1- 1/\beta \right) \kappa_2| \grad y \|_{\Ltwo}^2 + c \right) | \delta_b^2 |_{\Hs}^2 \\ \geq  (1-\beta) \kappa_1 \| \delta_b^1 \|_\Ltwo^2 + \left(\left(1- 1/\beta \right) \kappa_2| \grad y \|_{\Ltwo}^2 + c \right) D \| \delta_b^2 \|_{\Hs}^2.
\end{multline*}
The estimate \eqref{eq:seminorm_L2statecontrol} with $\kappa \coloneqq \min \left\{ (1-\beta) \kappa_1, \left(1- 1/\beta \right) \kappa_2| \grad y \|_{\Ltwo^N}^2 + c \right\}$ now holds for $\beta \in (0,1)$ chosen sufficiently close to $1$ in order to have $\kappa > 0$.

\item Using H\"older's inequality and  \Cref{lem:boundedness_of_solution_op_lin_pde} we obtain
\begin{multline*}
\int_\Omega \delta_b \cdot (\lambda \grad \delta_y ) \dx \leq \| \delta_b \|_\Lq \| \lambda \|_\Lp \| \grad \delta_y \|_{\Ltwo} \\ \leq C_{s,q} D \| \lambda\|_\Lp  \| \delta_b \|_\Hs^2 
 \end{multline*}
where $D \coloneqq C \| \grad y \|_\Ltwo^N$ for $C$ as in \Cref{lem:boundedness_of_solution_op_lin_pde}, and $C_{s,q}$ denotes the embedding constant of $\Hs$ into $\Lq$. Taking $c = 2  \alpha$ and $\kappa > 0$ such that \eqref{eq:seminorm_L2statecontrol} holds, we compute
\begin{multline*}
D^2_{(y,b)}L(\sigma,\bar{y},\bar{b},\bar{\lambda}) [(\delta_y, \delta_b),(\delta_y, \delta_b)]  \\ = \| \delta_y \|_\Ltwo^2 + \alpha | \delta_b |_{\Hs}^2 + \int_\Omega \delta_b \cdot (\bar{\lambda} \grad \delta_y ) \dx \\ \geq (1/2) \| \delta_y\|_\Ltwo^2 + \left(\kappa/2 - C_{s,q} D \|\lambda\|_\Lp\right) \| \delta_b \|_\Hs^2.
\end{multline*}
This proves the claim for $r = \kappa /(2 C_{s,q} D)$.
\end{enumerate}
\end{proof}

\begin{proposition}
Let $(N,s)$ satisfy \ref{assumption:convection_relation_dim_s},  and let $p,q>1$ be such that $1/p+1/q = 1/2$ and $\Hs$ is continuously embedded in $\Lq$. Assume that there exists $(y,b)\in \Fad$ such that \eqref{eq:nec_suff_condition_bilevel} holds with strict inequality. Let $(\bar{y}, \bar{b} , \bar{\lambda})$ be a KKT point of \eqref{prob:conv_simplified}. Then there exists a constant $C > 0$ such that if 
\begin{equation*}
\| \bar{\lambda} \|_\Lp < C,
\end{equation*}
then $(\bar{y}, \bar{b})$ is the unique solution to \eqref{prob:conv_simplified}.

\end{proposition}
\begin{proof}
It follows from  \Cref{prop:existence_of_solutions_llp_advection} that  \eqref{prob:conv_simplified} has a solution. A  look at the proof of \Cref{prop:existence_of_solutions_llp_advection} reveals that under our assumption all solutions to \eqref{prob:conv_simplified} must be contained in a norm ball with finite radius in $\Hen \times \Hs$. In the following we denote this ball by $B$. The proof is heavily based on the observation that
\begin{equation}\label{eq:e_bilinear_diff}
e_y(\bar{y},\bar{b}) [y-\bar{y}] + e_b(\bar{y}, \bar{b}) [b - \bar{b}] + e_{yb}[y-\bar{y},b-\bar{b}] = 0 \quad \text{ in } \Hend,
\end{equation}
for all $(y,b) \in \Fad$, which in turn follows from the fact that $e$ is the sum of a bilinear and an affine linear function. Moreover, we use that since $f \neq 0$  as a consequence of \cite[Theorem 2.1]{CasasMateosRoesch2019} we know that $\bar{y} \neq 0$. The proof itself is divided into three steps.
\begin{enumerate}
\item Let $b \in \Hs \setminus \{\bar{b}\}$ be such that $b - \bar{b}$ is constant and let $y \in \Hen$ be the unique element in $\Hen$ such that $e(y,b) = 0$. We claim that $y \neq \bar{y}$. To prove this, we argue by contradiction: If $y = \bar{y}$, then \eqref{eq:e_bilinear_diff} implies that
\begin{equation*} 
\left(b - \bar{b}\right) \cdot \grad \bar{y} = 0 \quad \text{in } \Hend.
\end{equation*}
Since by assumption $b-\bar{b}$ is a nonzero and constant, by \Cref{prop:partial_derivative_of_Hen_zero_implies_zero} it follows that $\bar{y} = 0$. This is a contradiction.

\item We now show that for arbitrary $c > 0$ there exists $\kappa > 0$ such that for all $(y,b) \in \Fad \cap B $ we have
\begin{equation}\label{eq:bound_lower}
\| y - \bar{y} \|_\Ltwo^2 + c | b - \bar{b} |_\Hs^2 \geq \kappa \|b-\bar{b}\|_\Ltwo^2. 
\end{equation}
We begin by proving that there exist $\kappa_1 > 0$ and $M>0$ such that  \eqref{eq:bound_lower} with $\kappa = \kappa_1$ holds for all $(y,b) \in \Fad \cap B$ with $\| b - \bar{b} \|_\Ltwo < M$. To do this, we argue by contradiction: If this is false, then there is a sequence $(y^n, b^n)$ in $(\Fad \cap B) \setminus \{(\bar{y}, \bar{b})\}$ with $\| b^n - \bar{b} \|_\Ltwo \to 0$ as $n \to \infty$ and a nullsequence $(\kappa^n)$  in $(0,\infty)$ such that
\begin{equation}\label{eq:uniqueness_cont_eq}
\| y^n - \bar{y} \|_\Ltwo^2 + c | b^n - \bar{b}|_\Hs^2 < \kappa^n \| b^n - \bar{b}\|_\Ltwo^2.
\end{equation}
It follows that $\| b^n - \bar{b} \|_\Hs \to 0$ as $n \to \infty$. We define a sequence $(\delta_b^n)$ in $\Hs^N$ by setting $\delta_b^n \coloneqq (b^n - \bar{b}) /\|b^n - \bar{b}\|_\Ltwo$ for every $n \in \N$. By definition we have $\| \delta_b^n \|_\Ltwo = 1$. Moreover, it follows from \eqref{eq:uniqueness_cont_eq} that $| \delta_b^n |_\Hs \to 0$ as $n \to \infty$. Hence $(\delta_b^n)$ has a subsequence, which converges weakly in $\Hs$ and strongly in $\Ltwo$ to a constant and nonzero element $\delta_b$ of $\Hs$. Additionally, for the sequence $(\delta_y^n)$ defined by $\delta_y^n \coloneqq (y^n - \bar{y}) /\|b^n - \bar{b}\|_\Ltwo$ we have
\begin{equation*}
\| \delta_y^n \|_\He^2 \leq \frac{C^2\left(\| b^n - \bar{b} \|_\Ltwo^2 + | b^n - \bar{b}|_\Hs^2\right)} {\|b^n - \bar{b} \|_\Ltwo^2}  \leq C^2 (1+ \kappa^n/c),
\end{equation*}
where $C>0$ is the Lipschitz constant of the solution operator $b \mapsto y(b)$ on $B$. Consequently, $(\delta_y^n)$ has a weak cluster point $\delta_y$ in $\Hen$. Let $\phi \in \Testfunctions$ be arbitrary and $\phi^n \coloneqq \phi / \| b^n - b \|_\Ltwo$. Testing \eqref{eq:e_bilinear_diff} for $(y,b)=(y^n,b^n)$ with $\phi^n$ we obtain 
\begin{equation}\label{eq:cont_12}
\int_\Omega \grad \delta_y^n \cdot \grad \phi + (\bar{b} \cdot \grad \delta_y^n) \phi + (\delta_b^n \cdot \grad \bar{y}) \phi  + [\delta_b^n \cdot \grad (y^n-y)] \phi \dx = 0.
\end{equation}
Taking the limit $n \to \infty$ in \eqref{eq:cont_12} and \eqref{eq:uniqueness_cont_eq} we deduce
that $(\delta_y, \delta_b) \in \ker D e(\bar{y}, \bar{b})$ and $\| \delta_y \|_\Ltwo = 0$. This is a contradiction since the mapping $\delta_b \mapsto \delta_y$ such that  $(\delta_y, \delta_b) \in \ker D e (\bar{y}, \bar{b})$ is injective on the space of constant functions. 

It remains to show that there exists $\kappa_2 > 0$ such that \eqref{eq:bound_lower} with $\kappa = \kappa_2$ holds for all $(y,b) \in \Fad \cap B$ with $\| b- \bar{b}\|_\Ltwo \geq M$. For  this, it suffices to prove that
\begin{equation}\label{eq:contr_2}
\min_{(y,b) \in \Fad \cap B \colon  \| b - \bar{b}\|_\Ltwo \geq M} \| y - \bar{y} \|_\Ltwo + c | b - \bar{b} |_\Hs^2
\end{equation}
has a solution, and that the optimal function value $\bar{\kappa}$ in \eqref{eq:contr_2} is strictly larger than zero, since then we have
\begin{equation*}
\| y - \bar{y} \|_\Ltwo^2 + c | b - \bar{b} |_\Hs^2 \geq \bar{\kappa} \geq (\bar{\kappa}/r^2) \| b - \bar{b} \|_\Ltwo^2 
\end{equation*}
for all $(y,b) \in \Fad \cap B$ with $\| b - \bar{b} \|_\Ltwo \geq M$, where $r \coloneqq \sup \{  \| b - \bar{b} \|_\Ltwo \mid (y,b) \in \Fad \cap B \}$. Proving existence of solutions to \eqref{eq:contr_2} is straightforward since
\begin{equation*}
\{ (y,b) \in \Fad \cap B \mid \| b - \bar{b} \|_\Ltwo \geq M \}
\end{equation*}
is weakly sequentially compact; a fact which can be easily proven using that $\Hs$ is compactly embedded in $\Ltwo$ and \Cref{cor:stability_of_conv_state}. It now follows from the first step that the optimal function value is positive, as claimed above.

\item Using \eqref{eq:e_bilinear_diff}, for every $(y,b) \in \Fad \cap  B $ we compute
\begin{multline}\label{eq:estimate_uniq_1}
| \bar{\lambda} e_y(\bar{y},\bar{b}) [y-\bar{y}] + e_b(\bar{y}, \bar{b})[b - \bar{b}] | =| - \bar{\lambda} e_{yb} [y-\bar{y}, b - \bar{b}] | \\ \leq C_{s,q} \|\bar{\lambda}\|_\Lp \| y- \bar{y} \|_{\He} \| b- \bar{b} \|_{\Hs} \leq C_{s,q} D \|\bar{\lambda}\|_\Lp \| b- \bar{b}\|_{\Hs}^2,
\end{multline}
where $C_{s,q}$ is the embedding constant of the embedding of $\Hs$ into $\Lq$ and $D>0$ is the Lipschitz constant of the solution operator $b \mapsto y(b)$ on $B$. Let $F(y,b) \coloneqq  \| y - \ydelta \|_\Ltwo^2 + \alpha | b |_\Hs^2$ for every $(y,b) \in \Fad$. Moreover, let $c = 2\alpha $ and let $\kappa$ be the corresponding constant in \eqref{eq:bound_lower}. We have
\begin{multline*}
F(y,b) - F(\bar{y},\bar{b}) = F_y(\bar{y},\bar{b})(y-\bar{y}) + F_b(\bar{y}, \bar{b}) (b-\bar{b})+ \|y-\bar{y}\|_{\Ltwo}^2 + \alpha |b-\bar{b}|_{\Hs}^2  \\ = \bar{\lambda} [e_y(\bar{y}, \bar{b}) (y - \bar{y}) + e_b ( \bar{y}, \bar{b} ) (b- \bar{b}) ] + \|y-\bar{y}\|_{\Ltwo}^2 + \alpha |b-\bar{b}|_{\Hs}^2  \\ 
 \geq (1/2) \| y - \bar{y} \|_\Ltwo^2 +  \left(\kappa/2 -C_{s,q} D  \| \bar{\lambda} \|_\Lp \right) \| b- \bar{b} \|_\Hs^2.
\end{multline*}
Here we used that $(\bar{y}, \bar{b}, \bar{\lambda})$ is a KKT point for the second equality and the estimates in \eqref{eq:bound_lower} and \eqref{eq:estimate_uniq_1} for the last inequality. It follows that if
\begin{equation*}
\| \bar{\lambda} \|_\Lp < \kappa / (2C_{s,q} D),
\end{equation*} 
then $(\bar{y},\bar{b})$ is the unique solution to \eqref{prob:conv_simplified}.
\end{enumerate}
\end{proof}

\section{Numerical experiments}
In this section we present results of numerical experiments where we solve the learning problem for the linear forward problem from \Cref{ex:poissoneq}. Here we let $\Omega = (0,1)$ and $\rho =0.1$. As a regularization operator for the weights, we consider the particular choice
\begin{equation*}
R(\sigma) = \beta \int_0^d \sigma \, \dx  + \alpha | \sigma |_{L^2((0,d))}^2 \quad \text{for } \sigma \in \Wad, 
\end{equation*}
where $\alpha,\beta > 0$. The obtained results are compared to results obtained for choosing the optimal regularization parameter $\nu$ in \eqref{prob:reg_Sobolev_seminorm} by solving a similar learning problem. 

\subsection{Data}  
We let $\omega = (\omega_1, \dots, \omega_m)^T$ be an $m$-dimensional random variable following a uniform distribution on $[0,1]^m$. We distinguish between two cases. 
\begin{enumerate}[label=(\Alph*)]
\item In the first case,  we let $m=3$ and \begin{equation*}
\uex(x, \omega) = \sin(20\omega_1x)+\omega_3\cos(40\omega_2x) \quad \text{for } x \in (0,1).
\end{equation*}\label{item:creation_uex_caseA}
\item In the second case,  we let $m=3$ and
\begin{equation*}
\uex(x, \omega) =  3 \omega_3\cos(6\pi x+10\omega_1)+ 2 \omega_2 \quad \text{for } x \in (0,1).
\end{equation*}\label{item:creation_uex_caseB}
\end{enumerate}
To create data for training and validation, we take samples $\omega^i$ from $\omega$ and let $u^{\dagger,i} = \uex( \cdot, \omega^i)$. We discretize the problem using linear Lagrange elements for equidistant grid points $0 = x_1 < \cdots < x_{N_E} = 1$, where $N_E=128$. The corresponding (discrete) ground truth state $y^{\dagger,i}$ is computed by solving the discretized forward problem. Noisy data measurements are generated by point wise setting
\begin{equation*}
\ydelta^i(x_j) = y^{\dagger,i}(x_j) + \epsilon \,\xi_{i,j},
\end{equation*}
where $\xi_{i,j} \in \R$ are samples drawn from a normally distributed random variable with mean $0$ and standard deviation $1$, and $\epsilon$ is the noise level. In order to discretize the weights, we use piecewise constant FEM. We let $(u_1,\dots, u_{N_E})$ and $(\sigma_1,\dots, \sigma_{N_E+1})$ denote a basis for the control and the weight FEM spaces, respectively.  The integrals
\begin{equation*}
\m{L}^s(\sigma_k) u_i u_j = \iint\limits_{\Omega \times \Omega}  \frac{(u_i(x) - u_i(x))(u_j(x) - u_j(x))}{|x-y|^{1+2s}} \sigma_k(|x-y|) \dx \dy  \\
\end{equation*}
are computed analytically using symbolic integration. 
\subsection{Applied methods}
Recall that the lower level problem has a unique solution for every regularization weight $\sigma \in \Wad$. Using this, we define the reduced cost functional $F \colon  \Wad \to \R$ by 
\begin{equation*}
F(\sigma) \coloneqq \frac{1}{2} \| u(\sigma) - \uex \|_\Ltwo^2 + R(\sigma), \quad \text{for every } \sigma \in \Wad,
\end{equation*}
where $u(\sigma)$ is the unique solution to the lower level problem with weight $\sigma$. The learning problem \eqref{prob:bilevelsimple} can then be written as follows
\begin{equation}\label{eq:redbilevel1}
\min_{\sigma \in \Wad} F(\sigma) \quad \text{subject to} \quad \sigma_{\min} \leq \sigma \leq \sigma_{\max}
\end{equation}
where  $\sigma_{\min} \equiv \gamma_1 \chi_{[0,\delta]}$ and $\sigma_{\max} \equiv \gamma_2$. A necessary optimality condition for $\sigma^*$ to be a solution of $\eqref{eq:redbilevel1}$ is
\begin{equation}\label{eq:optimality_reduced1}
\langle F' (\sigma^*), \sigma -  \sigma^* \rangle \geq 0 \quad \text{for all } \sigma \in \Wad,
\end{equation}
If $F'(\sigma^*)$ has a Riesz representative $\grad F(\sigma^*)$ in $\Ltwo$, then \eqref{eq:optimality_reduced1} is equivalent to
\begin{equation*}
\sigma^* = P_{\Wad}( \sigma^* - c \grad F (\sigma^*)),
\end{equation*}
for arbitrary $c > 0$, where $P_\Wad$ is the $L^2$-minimal projection on $\Wad$. We define   
\begin{equation*}
\Phi(\sigma) \coloneqq \sigma - P_{\Wad}(\sigma - c \grad F(\sigma)),
\end{equation*}
which can be interpreted pointwise almost everywhere on $(0,d)$ as
\begin{equation*}
\Phi(\sigma)(x) = \sigma(x) - \max\left[\sigma_{\min}(x), \min[\sigma_{\max}(x),\sigma(x) - c \grad F(\sigma)(x)]\right].
\end{equation*}
In order to solve the reduced learning problem we use a non-linear primal-dual active set method provided in \cite{Ito2004}. To solve the unconstrained problems on the inactive set we use a globalized quasi-Newton method accompanied by an Armijo line search (compare \cite[algorithm 11.5 on p 60]{Ulbrich2012}).

Strictly speaking the convergence analysis provided in \cite{Ito2004} does not apply to our setting. In practice the algorithm performed satisfactorily.

\subsection{Results}
We tested the algorithm in MATLAB for various choices of $s$. We create $N_{\text{train}}$ training and $N_{\text{val}}$  validation data vectors. The training set is divided into $N_{\text{batch}}$ training batches. Each training batch then consists of $\text{batchsize} = N_{\text{train}}/N_{\text{batch}}$ training vectors. For $1 \leq i \leq N_{\text{batches}}$ an optimal regularization weight $\sigma^{*,i}$ is computed for the i-th batch by solving the associated learning problem. Subsequently, the optimal weights are tested on the validation set. Thus for each validation vector $(\yex, \uex, \ydelta)$ and each optimal weight $\sigma^{*,i}$ we compute a solution $u_{\sigma^{*,i}}$ to the corresponding lower level problem. We then compute the validation error given by $\| u_{\sigma^{*,i}} - \uex \|_\Ltwo^2$. The average validation error is obtained by averaging the validation error over all validation vectors and weights $\sigma^{*,i}$. We then repeat the same training and validation procedure, but instead of the optimal weight, we only learn the optimal regularization parameter $\nu$ for regularization with a fractional order Sobolev seminorm (corresponding to a weight $\sigma \equiv 1$ ). The obtained training and validation errors for different batchsizes are provided in \Cref{table:training_validation_errorsA,table:training_validation_errorsB}. We notice the following behaviour:
\begin{enumerate}[label=\arabic*.)]
\item In all tested cases, both the training and the validation error for the optimal weight $\sigma$ are smaller than the training and validation error for the optimal regularization parameter $\nu$ (see \Cref{table:training_validation_errorsA,table:training_validation_errorsB}). 
\item Overall, the benefits of being able to choose a distance dependent weight $\sigma$ over choosing only a scalar regularization parameter were less pronounced for larger values of $s$ than for smaller ones (compare \Cref{table:A:a} with \Cref{table:A:b} and \Cref{table:B:a} with \Cref{table:B:b}). 

\item Note that for $s=0.1$ the optimal weight in case \ref{item:creation_uex_caseB} has distinct peaks around $1/3$,  $2/3$, and close to $1$ (see \Cref{fig:fig_weights}). This can be explained by the fact that functions created as in \ref{item:creation_uex_caseB} are always periodic with a period $1/3$.

\item In case \ref{item:creation_uex_caseB} the influence of the weight was much larger compared to case \ref{item:creation_uex_caseA}. The validation error was significantly decreased for $s = 0.1$ (see \Cref{table:B:a}). We attribute this to the fact that in case \ref{item:creation_uex_caseB} both the training and validation functions were periodic with the same period. This constitutes a case where in our opinion the impact of being able to choose a nonlocal weight is clearly visible.  

\item Large batch sizes improve estimates for $\nu$ as well as $\sigma$. In fact, the results for smaller batch sizes (even after taking  several smaller batches involving the same amount of training data in total) can not reach the results obtained for one batch consisting of the total training set (compare the rows).

\end{enumerate}
In general, whether the additional computational effort when using fractional order regularization is justified, depends on the structure of the data. It should be noted that the significant improvement reported in case \ref{item:creation_uex_caseB} is not surprising. In fact, case \ref{item:creation_uex_caseB} was intentionally designed to provide an example where one would expect that being able to choose a distance dependent weight improves the reconstruction quality.

\begin{table}[tbhp]
{\footnotesize
\captionsetup{position=top}
\caption{Average training and validation error for optimal regularization parameter $\nu^*$ (second and third column) and optimal weight $\sigma^*$ (fourth and fifth column) in case \ref{item:creation_uex_caseA}. The training set and the validation set both consisted of 512 data vectors.}\label{table:training_validation_errorsA}
\begin{center}
\subfloat[$s = 0.1$\label{table:A:a}]{
\begin{tabular}{| c | c | c | c | c |} \hline batchsize & train error $(\nu^*)$ & val error $(\nu^*)$ & train error $(\sigma^*)$ & val error  $(\sigma^*)$ \\ \hline 8 & \num{1.805429e-02 } & \num{2.193907e-02 } & \num{1.414873e-02 } & \num{2.043156e-02 }\\ \hline 64 & \num{1.879361e-02 } & \num{2.105682e-02 } & \num{1.609630e-02 } & \num{1.858609e-02 }\\ \hline 512 & \num{1.897812e-02 } & \num{2.085729e-02 } & \num{1.651916e-02 } & \num{1.822781e-02 }\\ \hline \end{tabular} 
}

\subfloat[$s=0.9$\label{table:A:b}]{
\begin{tabular}{| c | c | c | c | c |} \hline batchsize & train error $(\nu^*)$ & val error $(\nu^*)$ & train error $(\sigma^*)$ & val error $(\sigma^*)$ \\ \hline 8 & \num{1.528922e-02 } & \num{1.975799e-02 } & \num{1.510885e-02 } & \num{1.965428e-02 }\\ \hline 64 & \num{1.626830e-02 } & \num{1.851232e-02 } & \num{1.612274e-02 } & \num{1.837110e-02 }\\ \hline 512 & \num{1.653546e-02 } & \num{1.820064e-02 } & \num{1.639905e-02 } & \num{1.804662e-02 }\\ \hline \end{tabular} 
}
\end{center}
}
\end{table}

\begin{table}[tbhp]
{\footnotesize
\captionsetup{position=top}
\caption{Average training and validation error for optimal regularization parameter $\nu^*$ (second and third column) and optimal weight $\sigma^*$ (fourth and fifth column) in case \ref{item:creation_uex_caseB}. The training set and the validation set both consisted of 512 data vectors. }\label{table:training_validation_errorsB}
\begin{center}
\subfloat[\label{table:B:a}$s = 0.1$]{
\begin{tabular}{| c | c | c | c | c |} \hline batchsize & train error $(\nu^*)$ & val error $(\nu^*)$ & train error $(\sigma^*)$ & val error $(\sigma^*)$ \\ \hline 8 & \num{2.029304e-02 } & \num{2.118399e-02 } & \num{1.020525e-02 } & \num{1.090874e-02 }\\ \hline 64 & \num{2.054934e-02 } & \num{2.085002e-02 } & \num{1.015662e-02 } & \num{1.040007e-02 }\\ \hline 512 & \num{2.059307e-02 } & \num{2.080284e-02 } & \num{1.136507e-02 } & \num{1.153481e-02 }\\ \hline \end{tabular} 
}

\subfloat[$s = 0.9$\label{table:B:b}]{
\begin{tabular}{| c | c | c | c | c |} \hline batchsize & train error $(\nu^*)$ & val error $(\nu^*)$ & train error $(\sigma^*)$ & val error $(\sigma^*)$ \\ \hline 8 & \num{1.365286e-02 } & \num{1.440280e-02 } & \num{1.330755e-02 } & \num{1.406108e-02 }\\ \hline 64 & \num{1.384132e-02 } & \num{1.416644e-02 } & \num{1.349480e-02 } & \num{1.382828e-02 }\\ \hline 512 & \num{1.386810e-02 } & \num{1.413716e-02 } & \num{1.352077e-02 } & \num{1.379932e-02 }\\ \hline \end{tabular} 
}
\end{center}
}
\end{table}
\begin{figure}[tbhp]
\centering
\subfloat[$s=0.1$]{
    \includegraphics[width=0.4\columnwidth]{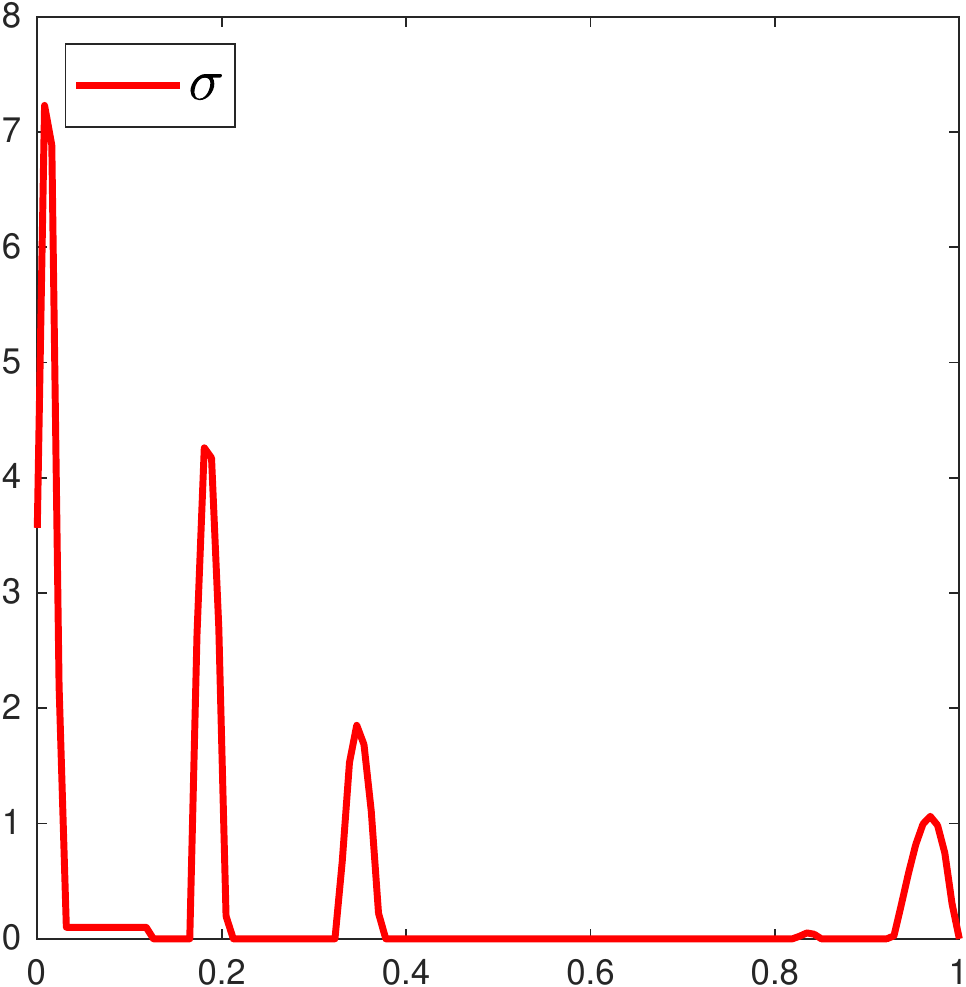}}
\subfloat[$s=0.9$]{
\includegraphics[width=0.4\columnwidth]{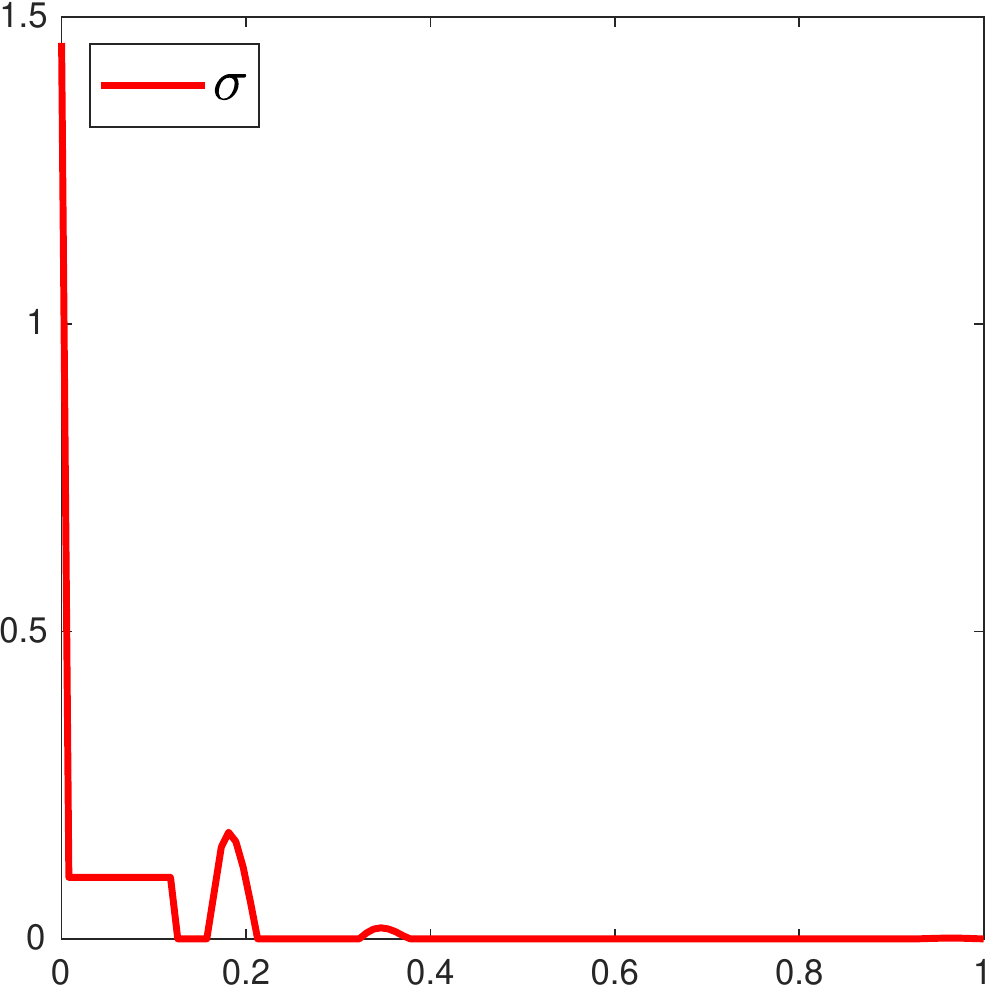}}

\subfloat[$s=0.1$]{\includegraphics[width=0.4\columnwidth]{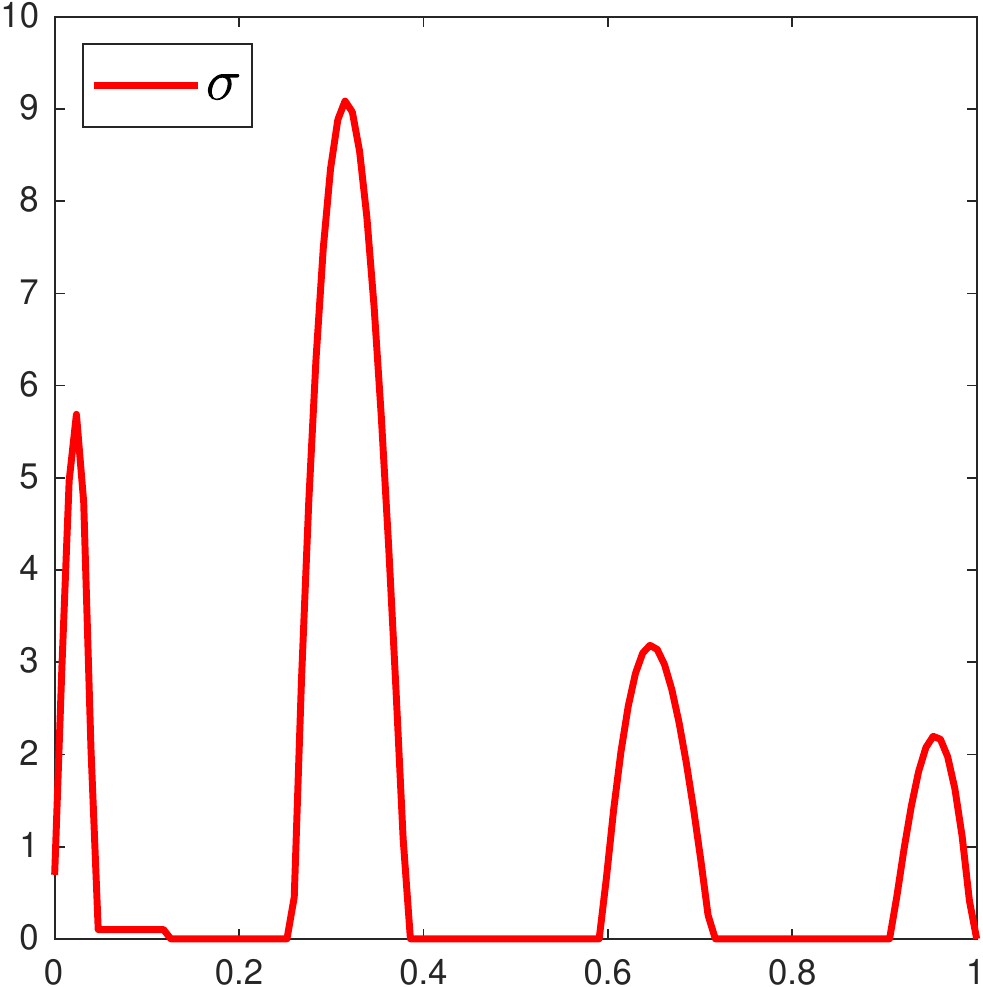}}
\subfloat[$s=0.9$]{\includegraphics[width=0.4\columnwidth]{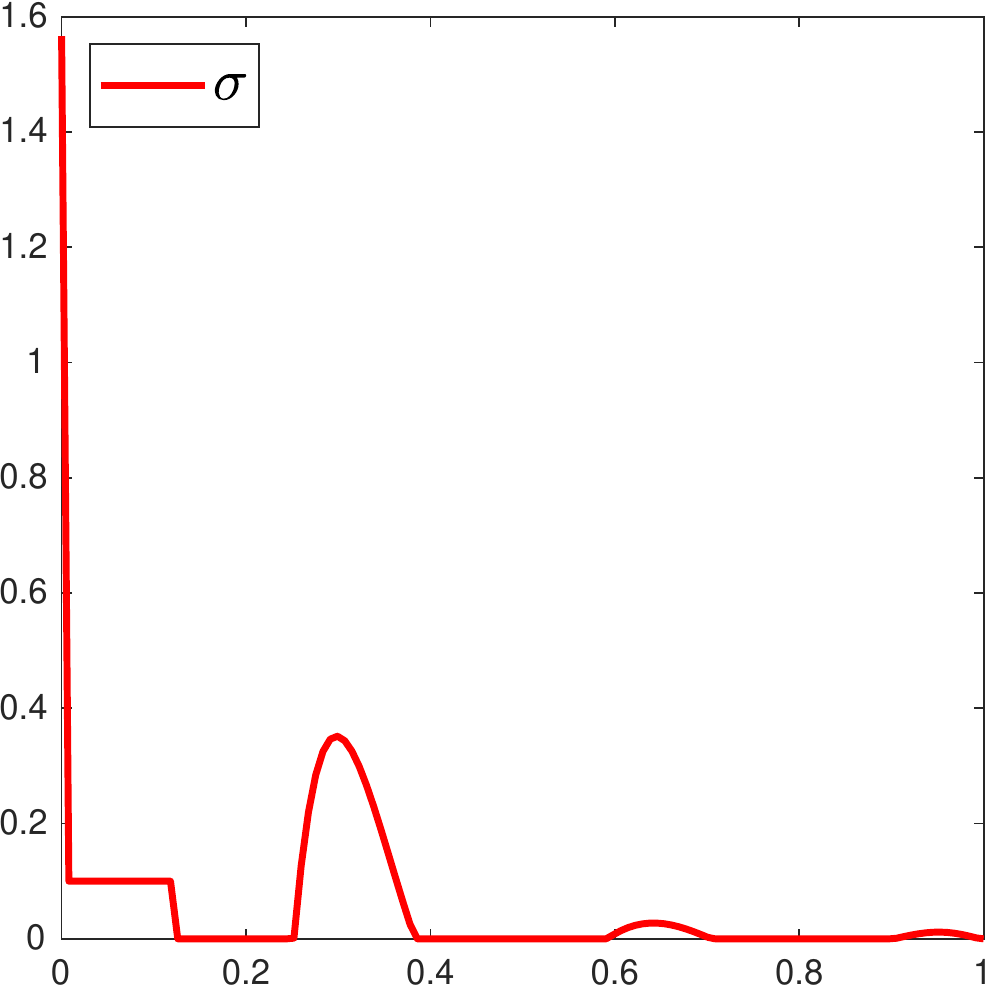}}

\caption{Optimal weights for linear state equation in case \ref{item:creation_uex_caseA} (first row) and case \ref{item:creation_uex_caseB} (second row). $1 \%$ additive noise was used. The training set consisted of 512 data vectors.}\label{fig:fig_weights}
\end{figure}

\begin{figure}[tbhp]
\centering
\subfloat[With optimal weight for $s=0.1$]{

    \includegraphics[width=0.4\columnwidth]{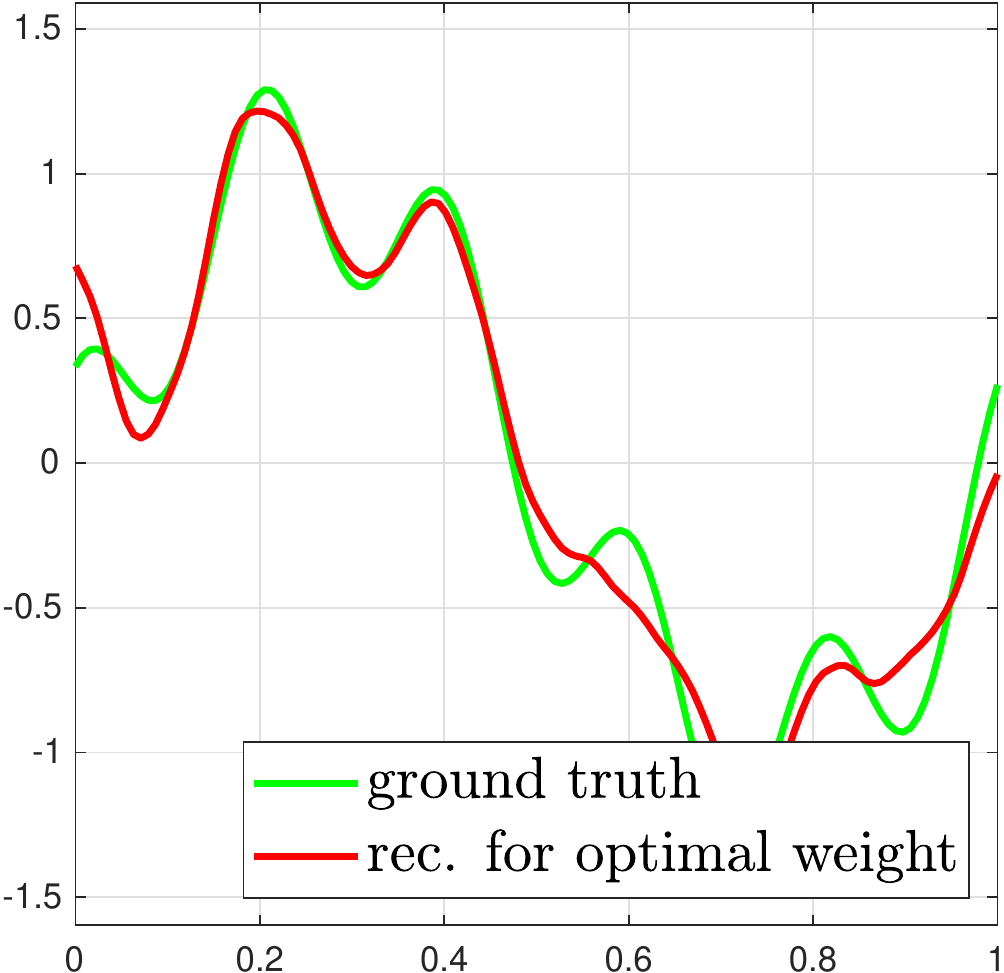}}
\subfloat[With optimal weight for $s=0.9$]{
    \includegraphics[width=0.4\columnwidth]{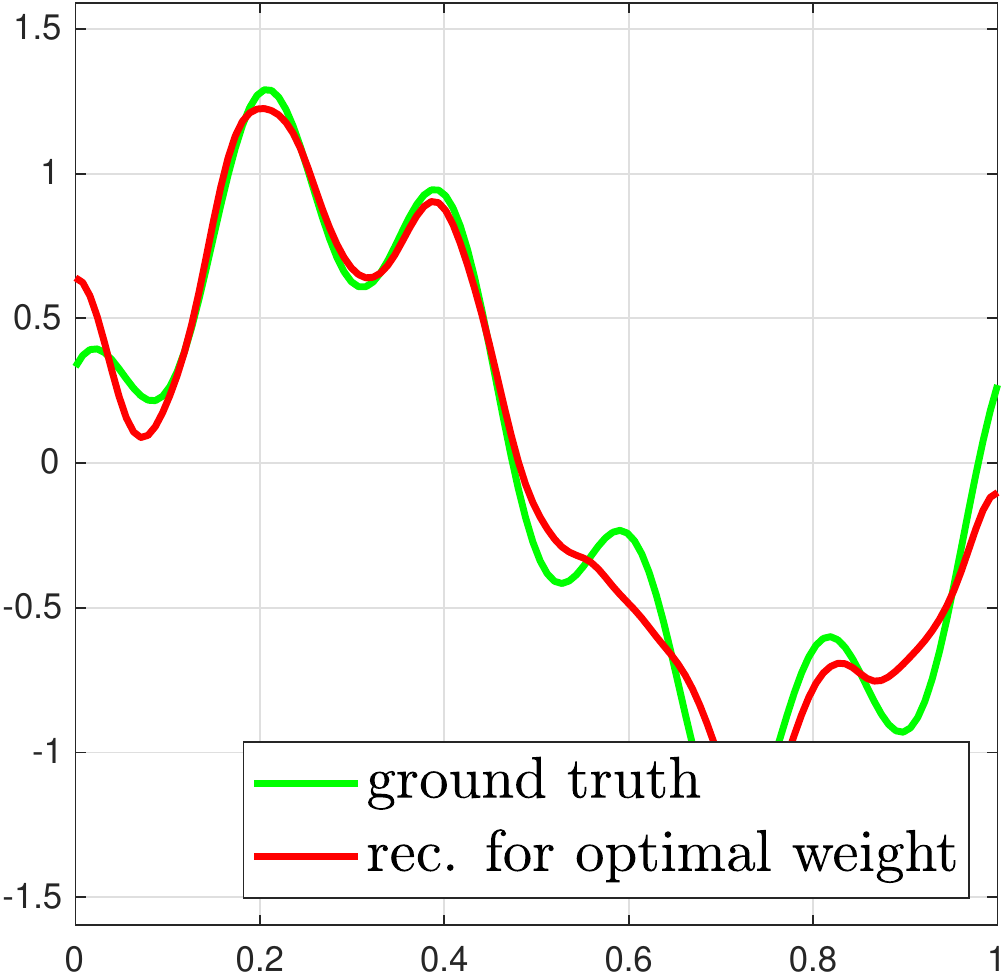}}

\caption{Ground truth and reconstructed controls for one data vector from the validation set for linear state equation in case \ref{item:creation_uex_caseA}. $1 \%$ additive noise was added to create noisy the measurements. The training and validation set both consisted of 512 data vectors.}
\end{figure}
\FloatBarrier

\bibliographystyle{abbrvurl}
\bibliography{extracted_bibliography.bib}
\end{document}